\documentclass[11pt,sort]{newarticle}
\usepackage[left=1.00in,top=1.00in,right=1.00in,bottom=1.00in]{geometry}
\usepackage{amssymb}
\usepackage{graphicx}
\usepackage{amsmath}
\usepackage{amsthm}
\usepackage{mathrsfs}
\usepackage{accents}
\usepackage{subfigure}

\newtheorem{corollary}{Corollary}
\newtheorem{theorem}{Theorem}

\newtheorem{proposition}{Proposition}
\newtheorem{remark}{Remark}

\theoremstyle{definition}
\newtheorem{definition}{Definition}

\newcommand{\RR}{\mathbb R}
\newcommand{\f}[1]{\mathbf{#1}}

\newcommand\Body{\Omega}
 
\newcommand{\disp}{w}
\newcommand{\dispTF}{v}
\newcommand{\force}{f}
\newcommand{\normal}{\f {n}}
\newcommand{\V}{\mathcal{V}}
\renewcommand{\S}{\mathcal{S}}

\newcommand{\btau}{\boldsymbol{\tau}}

\newcommand{\bnu}{\boldsymbol{\nu}}
\newcommand{\bn}{\boldsymbol{d}}
\newcommand{\bv}{\boldsymbol{v}}
\newcommand{\bx}{x,y}

\newcommand{\matrixSystem}{\mathrm{A}}
\newcommand{\rhs}{b}
\newcommand{\sol}{x}

\newcommand{\solBCContDer}{z}
\newcommand{\ConstraintDer}{\mathrm{C}_{1}}

\newcommand{\NSBCCont}{\mathrm{N}_0}
\newcommand{\NSBCContDer}{\mathrm{N}_1}
\newcommand{\Id}{I}

\begin{document}

\begin{frontmatter}
  
\title{  Analysis-suitable $G^1$ multi-patch
  parametrizations \\ for $C^1$ isogeometric spaces}

\author[1]{Annabelle Collin}
\author[1,2]{Giancarlo Sangalli}
\author[1]{Thomas Takacs}

\address[1]{Dipartimento di Matematica ``F. Casorati'', Universit\`a degli Studi di Pavia, Italy}
\address[2]{Istituto di Matematica Applicata e Tecnologie Informatiche
  ``E. Magenes'' (CNR), Italy}

\begin{abstract}
 One key feature of isogeometric analysis is that it allows  smooth shape
functions. Indeed, when isogeometric spaces are constructed from $p$-degree splines (and extensions,
such as NURBS), they enjoy up to $C^{p-1}$ continuity within  each  patch. However,
global continuity beyond $C^0$ on so-called multi-patch geometries
 poses some significant difficulties.  In this work, we consider
 planar multi-patch
 domains that have a parametrization which is  only $C^0$
 at the patch interface. On such domains we study  the 
 $h$-refinement of $C^1$-continuous  isogeometric
spaces. These spaces in general do not have  optimal approximation 
properties. The reason is that the $C^1$-continuity condition easily
over-constrains the solution which is, in the worst cases, fully
\emph{locked} to linears at the  patch interface. However, recently 
\cite{kapl-vitrih-juttler-birner-15} has given numerical evidence that optimal convergence
occurs for  bilinear two-patch geometries and cubic (or higher degree)
$C^1$ splines. This
is the starting point of our study. We introduce
the class of analysis-suitable $G^1$ 
geometry parametrizations, which includes piecewise bilinear
parametrizations.  We then
analyze the structure of $C^1$ isogeometric  spaces over analysis-suitable $G^1$ parametrizations and, by
theoretical results and numerical testing, discuss their 
approximation properties. We also consider examples of geometry
parametrizations that are not analysis-suitable,  showing
that in this case  optimal convergence of $C^1$ isogeometric  spaces is prevented.
\end{abstract}

\end{frontmatter}

\section{Introduction}
\label{intro}

Thanks to the use of smooth B-splines and NURBS, isogeometric  methods \cite{hughes2005isogeometric,cottrell2009isogeometric}
 have revitalized   the interest  for the use of  smooth approximating
functions for the numerical solution of partial differential equations. 
Advantages with respect to $C^0$ finite element methods are
improved accuracy and spectral
properties \cite{evans2009n,da2011some,hughes2008duality,Takacs:2015aa}
and the possibility to directly discretize differential operators of 
order higher than $2$. In the literature there are indeed many examples 
of isogeometric methods for $4^{th}$ order differential problems of
 relevant interest, such as Kirchhoff-Love plates/shells 
\cite{kiendl-bletzinger-linhard-09,benson2011large,da2012isogeometric}, the Cahn-Hilliard equation
\cite{gomez2008isogeometric}, and the Navier-Stokes-Korteweg equation 
\cite{gomez2010isogeometric}. 

Since higher dimensional spline spaces possess a tensor-product structure, the representation 
of domains that have a complex geometry is non-trivial. In this paper we focus on multi-patch representations. 
While the implementation of $C^0$-continuity over multi-patch domains is well understood (see e.g. \cite{Kleiss2012,da2014mathematical,scott2014isogeometric} for strong and \cite{brivadis2015isogeometric} for weak imposition 
of the $C^0$ conditions), $C^1$-continuity is not. Several studies have tackled the problem of constructing 
function spaces of $C^1$ or higher order smoothness. A first attempt to compare different ways to impose $C^1$-continuity 
in isogeometric analysis was presented in \cite{nguyen2014comparative}. We also refer to 
\cite{Buchegger2015,juttler2015isogeometric,wu2015bicubic} for $C^1$ smooth constructions for 
spline spaces and \cite{speleers2012isogeometric,lyche2014hermite} for triangulations, which can be seen as an alternative to the classical B-spline based isogeometric framework. 
Nevertheless, the construction of smooth isogeometric spaces with optimal
approximation properties on complex geometries is still an open and challenging
problem. This is related to the problem of finding 
parametrizations  of smooth surfaces having complex topology, which
is a fundamental  area of research in the community of Computer Aided
Geometrid Design (CAGD) over the last decades.

\begin{figure}[h]
  \centering
  \includegraphics[width=.35\textwidth]{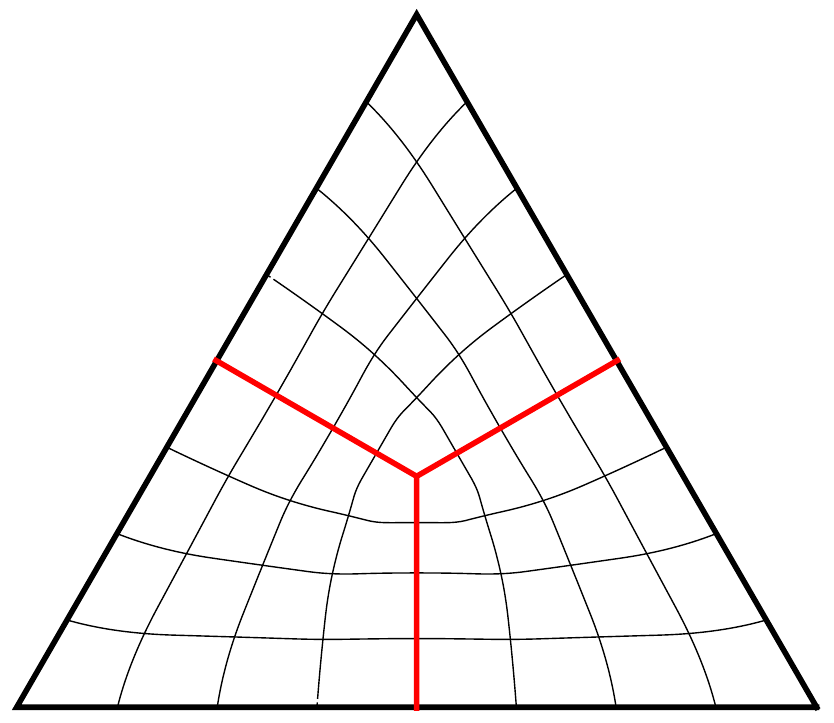}\hspace{.1\textwidth}
  \includegraphics[width=.35\textwidth]{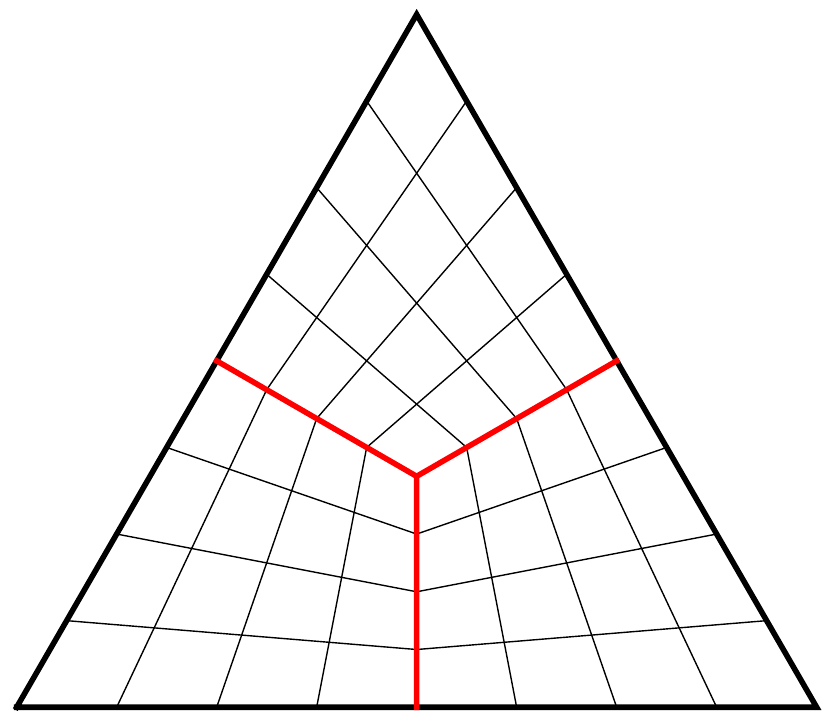}
 \caption{Two possible parametrization schemes: $C^1$ away from extraordinary points (left) and $C^0$ everywhere (right).}\label{fig:triangle-param-comparison}
\end{figure}

We review two
different strategies for constructing smooth multi-patch geometries and corresponding 
isogeometric spaces. One possibility is
to adopt a geometry parametrization which is globally smooth almost everywhere, 
with the exception of  a neighborhood of the extraordinary
points (or edges in 3D), see  Figure~\ref{fig:triangle-param-comparison} (left).  The other possibility is to use geometry parametrizations that are only
$C^0$ at patch interfaces, see  Figure~\ref{fig:triangle-param-comparison} (right).  The first option includes 
subdivision surfaces \cite{cirak2002integrated} and the T-spline
construction in \cite{scott2013isogeometric} and,  while
possessing attractive features, they  seem to possess optimal approximation properties for some configurations, see \cite{nguyen2016}, but in general lack accuracy
\cite{nguyen2014comparative,juttler2015isogeometric}. 
In our work we consider the second 
possibility, corresponding to the right part of Figure~\ref{fig:triangle-param-comparison}. 

The construction of  $C^1$  isogeometric functions over a $C^0$
parametrization can be interpreted conveniently as geometric  continuity $G^1$ of
the graph parametrization.
Bilinear multi-patch parametrizations of a  planar domain have been
analyzed in \cite{Matskewich-PhD,bercovier2014smooth},  where it is
shown that there exists a minimal determining set with  local degrees
of freedom for the space of (mapped) piecewise polynomial functions, with
global $C^1$ continuity,  if the polynomial  degree is high enough
($4$ if some additional conditions are fulfilled, $5$ in general). 
The recent preprint \cite{mourrain2015geometrically}  generalizes
the previous results, by  using advanced homology techniques, to arbitrary parametrizations, allowing
both triangular and quadrilateral patches. 

In the  work \cite{kapl-vitrih-juttler-birner-15}, the authors  consider splines
instead of polynomials within each patch,  construct a basis,
analyze  the space dimensionality of some configurations  and, for
the first time, perform numerical tests  to evaluate the order of
convergence when each patch is  $h$-refined. We recall that within the isogeometric
 framework, the concept of $h$-refinement, equivalent to knot
 insertion, 
is one of the three  constructions to increase accuracy of the spline
spaces (see
\cite{cottrell2009isogeometric}).  Remarkably,
\cite{kapl-vitrih-juttler-birner-15}  gives numerical evidence of 
optimal convergence for  $C^1$ splines of degree $3$ (or higher), on a two-patch bilinear geometry.  The 
authors also  show an example of \emph{over-constrained}  $C^1$
isogeometric  spaces, corresponding to a two-patch  non-bilinear
geometry parametrization. We refer to the latter situation as \emph{$C^1$ locking}.
Our  work develops the underlying theory.  While the previous papers
give explicit charaterizations in the form of minimal determining
sets \cite{Matskewich-PhD,bercovier2014smooth} or basis constructions
\cite{kapl-vitrih-juttler-birner-15,kapl-buchegger-bercovier-juttler-16},
 we use  an implicit characterization of the continuity conditions and
derive from it  information on the structure of the isogeometric
 space. 
As in \cite{kapl-vitrih-juttler-birner-15}, our interest is in the
 impact of $h$-refinement.  We study
$h$-refinement for arbitrary degree and regularity,  both
theoretically and numerically,  and point out its 
performance depending on the geometry parametrization.

 We set up our notation in Section \ref{preliminaries}. In Section
\ref{sec:C1-isogeometric-spaces} we define the class of analysis-suitable (AS) $G^1$ geometry parametrizations, which includes the
bilinear ones  and the extensions presented in Section 3.4 of 
\cite{kapl-vitrih-juttler-birner-15}. 
Then, in Section~\ref{two-patch-AS}, we study the structure of $C^1$ isogeometric
spaces over AS $G^1$ two-patch geometries.  Here, we give an
explanation of the optimal convergence of  $p$-degree isogeometric functions,   having $C^1$ continuity
across the patch interface and up to $C^{p-2}$ continuity within  the
patches. Furthermore, we discuss why  $C^1$ locking occurs for $C^{p-1}$
continuity within  the patches.  Note that in this paper we do not
derive  explicit error estimates, which will be the topic of a
future paper. In Section~\ref{two-patch-GENERAL} we
analyze $C^1$ isogeometric spaces constructed over more general
geometry parametrizations and conclude that $h$-convergence is suboptimal  beyond AS
$G^1$  geometries.  The extensions to surface domains and to NURBS are briefly discussed in Sections
\ref{surfaces} and \ref{sec:beyond-splines}, respectively. Numerical
tests on two-~and~multi-patch domains  are reported in
Section~\ref{numerical_tests}.  
There we present a significant example:  the multi-patch
parametrization of a smooth simply-connected planar domain. The
question of  existence and 
construction methods for AS $G^1$ multi-patch  parametrizations of
arbitrary  geometries remains to be studied.  
We summarize our results and draw conclusions in Section~\ref{conclusions}.
 
\section{Planar multi-patch spline parametrizations and isogeometric spaces}
\label{preliminaries}

Given an interval or a rectangle $\omega$,  we denote by  $\mathcal{S}^{p}_{r} (\omega)$ the spline space of
degree $p$ (in each direction) and continuity $C^r$ at all interior
knots. The knot  mesh in the parametric domain is assumed to be
uniform, with  mesh-size $h$ (which is  not explicitly
indicated in the notation) and interior knot multiplicity $p-r$. 
We write $\mathcal{S}^{p}_{r} $ instead of
$\mathcal{S}^{p}_{r} (\omega)$ when the domain $\omega$ is obvious
from the context.  We allow $r \geq p$, which stands for  $C^\infty$
continuity, that  is,  the case of  global tensor-product
polynomials on $\omega$. In this case we also use the notation
$\mathcal{P}^{p} (\omega) = \mathcal{S}^{p}_{p} (\omega)= \mathcal{S}^{p}_{p+1} (\omega)=\ldots$. 

 Consider a planar multi-patch domain of interest
\begin{equation}\label{eq:multi-patch-Omega} 
	\Omega = \Omega^{(1)} \cup \ldots \cup \Omega^{(N)} \subset \RR^2,
\end{equation}
where the closed sets $ \Omega^{(i)}$ form a regular  partition
(with disjoint interior). For simplicity, we do not allow hanging nodes. Each
$\Omega^{(i)}$ is assumed to be a   spline patch, that is 
\begin{eqnarray}
	\f F^{(i)}&:&[0,1]\times[0,1]= \widehat\Omega\rightarrow \Omega^{(i)}, \label{eq:F}
\end{eqnarray}
where $\f F^{(i)}\in \mathcal{S}^{p}_{r} (\widehat\Omega)\times \mathcal{S}^{p}_{r} (\widehat\Omega)$. We assume
\begin{equation}
  \label{eq:r-geq-1}
  r \geq 1
\end{equation}
($r \geq p$ means we have   Bezier patches $\Omega^{(i)}$) and  we
assume the parametrizations are not singular, i.e., for all $ i$ and
for all $ (u,v) \in \widehat\Omega$, 
\begin{equation}\label{eq:non-singular-F}
\det  \left[ \begin{array}{ll}
           D_u \f F^{(i)} (u,v) & D_v \f F^{(i)}(u,v) \\
         \end{array}\right] \neq 0.
\end{equation}

 For the sake of simplicity we do not consider more general configurations, e.g., non-uniform knot meshes, different degree
  or continuity parameters in each parametric direction,
  different  continuity at different knots, or  different spline
  spaces for different patches. Indeed, our simple configuration
  already presents the key features and difficulties we are interested
  in.  

We  assume global  continuity of the patch 
parametrizations. This means the following. Let  us fix  $ \Gamma =
\Gamma^{(i,j)} =  \Omega^{(i)} \cap 
\Omega^{(j)}$. When using the superscript as $(i,j)$  in the paper, we assume
implicitly that $i$ and $j$ are such that $\Gamma^{(i,j)}$ is not a point or an
 empty set. Let $\f F^{(L)} $, $\f F^{(R)}$ be given  such that
\begin{equation}
  \label{eq:F-left-and-right}
  \begin{aligned}
    	\f F^{(L)}&: [-1,0]\times[0,1]=
        \widehat\Omega^{(L)}\rightarrow \Omega^{(L)} = \Omega^{(i)}, \\
	\f F^{(R)}&: [0,1]\times[0,1]=
        \widehat\Omega^{(R)}\rightarrow \Omega^{(R)} = \Omega^{(j)}, 
  \end{aligned}
\end{equation}
where $(\f F^{(L)})^{-1} \circ \f F^{(i) }$ and $(\f
F^{(R)})^{-1} \circ \f F^{(j) }$  are  linear transformations 
(combinations of a translation, rotation and symmetry). Moreover, the
parametrizations agree at $u=0$, i.e., there is a $\f F_0:[0,1] \rightarrow \RR^2$ with
\begin{equation}
  \label{eq:Gamma}
  \begin{aligned}
      \Gamma &= \{\f  F_0(v) = \f F^{(L)}(0,v) = \f F^{(R)}(0,v), v\in[0,1] \}.
  \end{aligned}
\end{equation}
An example is depicted in Figure~\ref{fig:F-left-and-right}.
\begin{figure}[h]
  \centering
     \includegraphics[trim=0  250 0 50, clip,width=\textwidth]{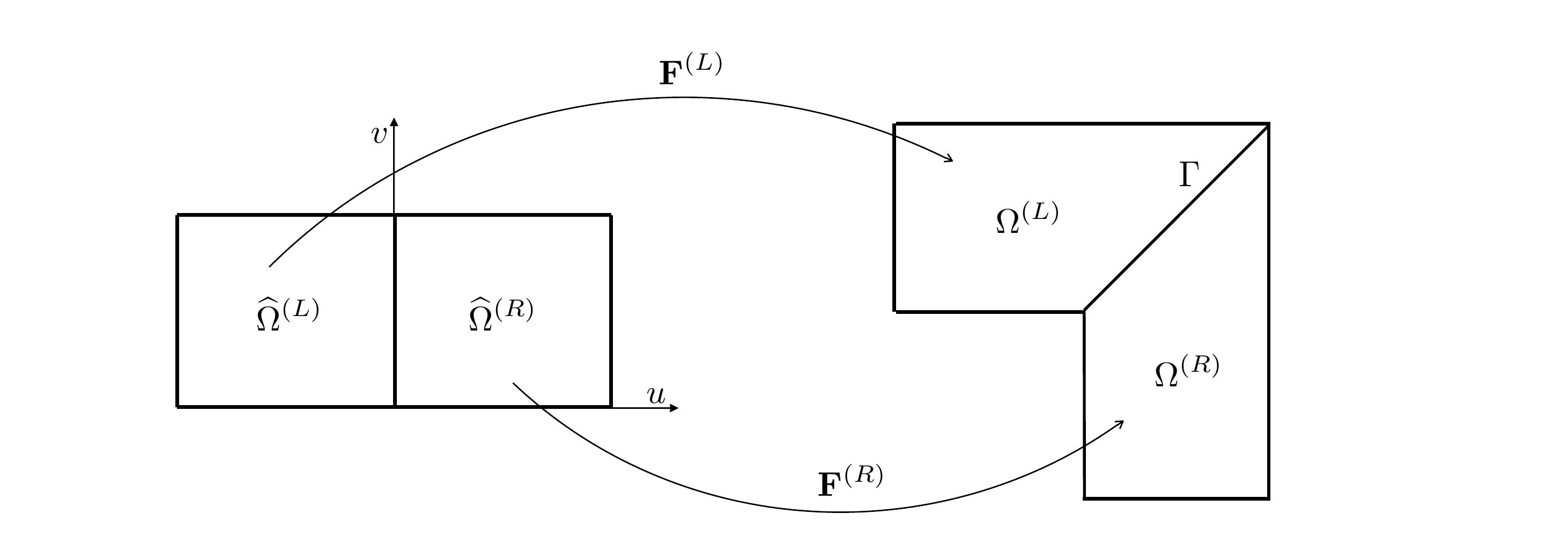}
  \caption{Example of the general setting of \eqref{eq:F-left-and-right}--\eqref{eq:Gamma}.}\label{fig:F-left-and-right}
\end{figure}

\begin{remark}
The domain  $\Omega = \Omega^{(1)} \cup \ldots \cup \Omega^{(N)}$
can be endowed with a spline manifold structure as defined, e.g.,  in
\cite{groisser2015matched,sangalli2015unstructured,mourrain2015geometrically}.
In the framework of \cite{sangalli2015unstructured}, each pair of adjacent subsets $
\Omega^{(i)}, \; \Omega^{(j)}$   is naturally  associated with a chart  $[-1,1]\times[0,1]=
        \widehat\Omega^{(L)} \cup  \widehat\Omega^{(R)}$ through  the
        maps $ \f F^{(L)}$ and $\f F^{(R)}$. \end{remark}
\begin{definition}[Isogeometric spaces]\label{defi:V0-V1}
  The  isogeometric space corresponding to $\mathcal{S}^{p}_{r}$ and $\Omega$ is given as 
  \begin{equation}
    \label{eq:V}
    \mathcal{V} = \left \{ \phi:\Omega \rightarrow
      \mathbb{R}\text{ such that } \phi \circ \f F^{(i)} \in
  \mathcal{S}^{p}_{r} (\hat \Omega) , i=1,\ldots,N \right \}. 
  \end{equation}
Furthermore we have
 \begin{equation} \label{eq:V0}
 \mathcal{V}^{0}=\mathcal{V}\cap C^0(\Omega),
  \end{equation}
and
 \begin{equation} \label{eq:V1}
 \mathcal{V}^{1}=\mathcal{V}\cap C^1(\Omega).
  \end{equation}
\end{definition}

The graph  $\Sigma \subset \Omega \times \RR$ of an isogeometric function  $\phi:\Omega \rightarrow
\mathbb{R}$ is naturally split into patches $\Sigma^{i}$ having the parametrizations 
\begin{equation}\label{eq:isogeom-function-parametrization}
  \left [
    \begin{array}{c}
      \f F^{(i)}\\ g^{(i)} 
    \end{array}
\right ]\, : \, [ 0,1]\times[0,1]= \widehat\Omega\rightarrow \Sigma^{(i)}
\end{equation}
where $g^{(i)} =\phi \circ \f F^{(i)} $. 

In order to analyze  the smoothness of an isogeometric function  along one
interface  $\Gamma = \Gamma^{(i,j)} =  \Omega^{(i)} \cap 
\Omega^{(j)}$, we introduce 
\begin{equation}
  \label{eq:Fg-left-and-right}
  \begin{aligned}
     \left [
    \begin{array}{c}
      \f F^{(L)}\\ g^{(L)} 
    \end{array}
\right ]&: [-1,0]\times[0,1]=
\widehat\Omega^{(L)}\rightarrow\Sigma^{(i)}=\Sigma^{(L)} , \\
 \left [
    \begin{array}{c}
      \f F^{(R)}\\ g^{(R)} 
    \end{array}
\right ]&: [0,1]\times[0,1]=  \widehat\Omega^{(R)}\rightarrow \Sigma^{(j)}=\Sigma^{(R)} , 
  \end{aligned}
\end{equation}
where $ g^{(L)} $, $ g^{(R)} $ are defined obviously  as  extensions of
\eqref{eq:F-left-and-right}, see Figure
\ref{fig:Fg-left-and-right}. Continuity of $\phi$ is implied by the
continuity of the graph parametrization, which we assume and set to
\begin{equation}
  \label{eq:g-continuity}
   g_0(v) = g^{(L)}(0,v) = g^{(R)}(0,v),
\end{equation}
for all $v\in[0,1]$, analogous to \eqref{eq:Gamma}.

\begin{figure}[h]
  \centering
  \includegraphics[trim=0  250 0 50, clip,width=\textwidth]{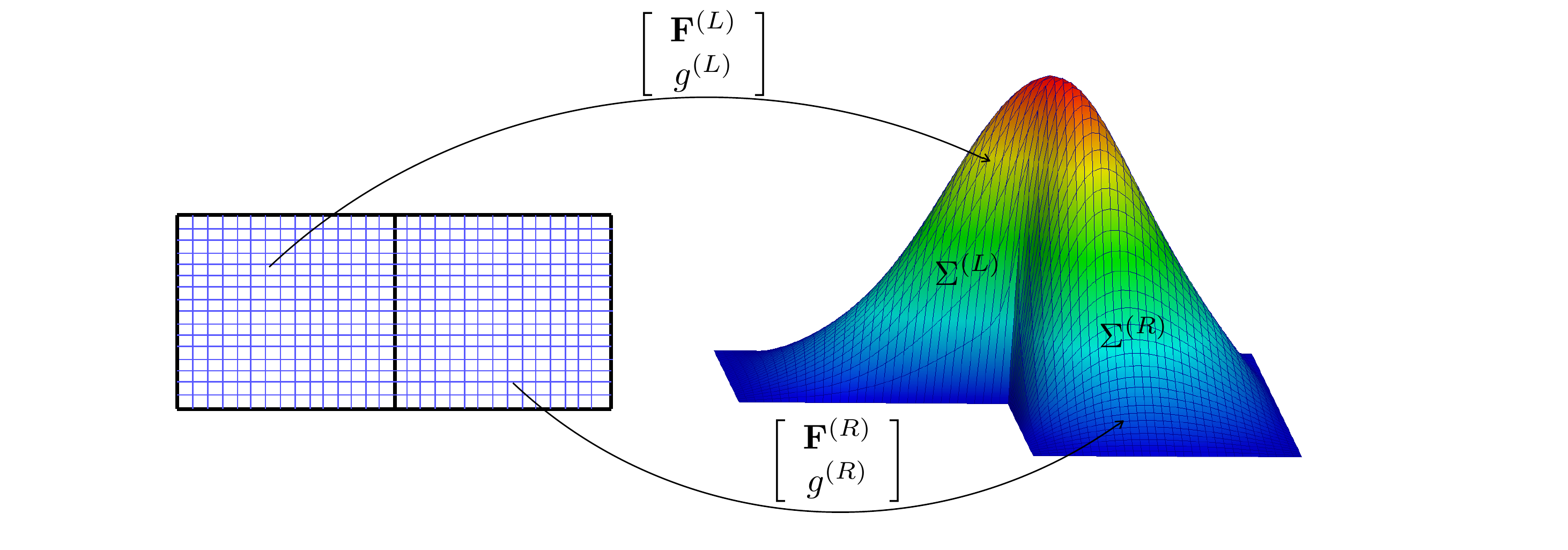}
 \caption{Example of the general setting of \eqref{eq:Fg-left-and-right}.}\label{fig:Fg-left-and-right}
\end{figure}

\section{$C^1$ isogeometric spaces}
\label{sec:C1-isogeometric-spaces}
If an isogeometric function is $C^1$ within each patch (condition
\eqref{eq:r-geq-1}), with non-singular parametrization (condition
\eqref{eq:non-singular-F}) and is globally continuous (condition
\eqref{eq:g-continuity}), then it is globally $C^1$ if and only if 
 there exists a  well defined  tangent plane  at 
each point of the interfaces $\Sigma^{(i)} \cap \Sigma^{(j)} $. 
Focusing on one interface, with notation \eqref{eq:Fg-left-and-right},
the tangent planes from the left and right sides are formed by the two
pairs of vectors
\begin{displaymath}
  \left [
    \begin{array}{c}
     D_u \f F^{(L)}(0,v) \\ D_u g^{(L)} (0,v) 
    \end{array}
\right ] \text{and} \left [
    \begin{array}{c}
      D_v  \f F_{0}(v) \\ D_v  g_{0}(v)
    \end{array}
\right ] \text{ as well as } \left [
    \begin{array}{c}
       D_u  \f F^{(R)}(0,v) \\    D_u g^{(R)} (0,v) 
    \end{array}
\right ] \text{and} \left [
    \begin{array}{c}
      D_v  \f F_{0}(v) \\ D_v  g_{0}(v)
    \end{array}
\right ] ,
\end{displaymath}
respectively. 
These are three different vectors (the vector tangent to $\Sigma^{(i)} \cap
\Sigma^{(j)} $ is in common) that form a unique tangent plane,
i.e. they are coplanar, if and only if they are  linearly dependent. In other words  the
isogeometric function $\phi$ is $C^1 $ on $\Omega^{(i)} \cup \Omega^{(j)} $
if and only if
\begin{equation}\label{eq:G1-for-F-det}
	\det \left[ \begin{array}{lll}
	D_u \f F^{(L)} (0,v) & D_u \f F^{(R)} (0,v) &  D_v  \f F_{0}(v)\\
	D_u g^{(L)} (0,v) & D_u g^{(R)} (0,v) &  D_v  g_{0}(v) 
	\end{array}
	 \right] = 0
\end{equation}
for all $v \in [0,1]$. In the  context of isogeometric methods,  the domain $\Omega$ and its
parametrization are  given at the first stage, then
\eqref{eq:G1-for-F-det} is  the condition on the
isogeometric function in parametric coordinates (i.e., $g^{(L)} $
and $g^{(R)}$)    that gives $C^1$ continuity of the isogeometric
function in physical coordinates.  

In CAGD literature, condition \eqref{eq:G1-for-F-det} is named \emph{geometric continuity
  of order $1$}, in short $G^1$,  and is commonly stated  as in the following Definition (see,
e.g., \cite{peters1991smooth,liu1989gc,beeker1986smoothing}).  

\begin{definition}[$G^1$-continuity at  $\Sigma^{(i)} \cap
  \Sigma^{(j)} $]\label{def:G1} Given the parametrizations $\f F^{(L)} $,  $\f F^{(R)}$,  $ g^{(L)} $, $ g^{(R)} $ as in 
\eqref{eq:F-left-and-right}, \eqref{eq:Fg-left-and-right}, fulfilling
\eqref{eq:r-geq-1},  \eqref{eq:non-singular-F} and
\eqref{eq:g-continuity},  we say that the graph parametrization is
$G^1$  at the interface $\Sigma^{(i)} \cap \Sigma^{(j)} $    if there exist $\alpha^{(L)}  : [0,1]
 \rightarrow \RR $, $ \alpha^{(R)} : [0,1] \rightarrow \RR $ and $\beta  : [0,1]
 \rightarrow \RR $ such that for all $  v \in [0,1]$,
 \begin{equation}
   \label{eq:alpha-sign-condition}
 \alpha^{(L)}  (v) \alpha^{(R)} (v) > 0
 \end{equation}
and 
 \begin{equation}
   \label{eq:G1-for-Fg-alpha-beta-gamma}
  	\alpha^{(R)} (v)  \left [
    \begin{array}{c}
     D_u \f F^{(L)}(0,v) \\ D_u g^{(L)} (0,v) 
    \end{array}
\right ]
-
   \alpha^{(L)}(v) 
\left [
    \begin{array}{c}
       D_u  \f F^{(R)}(0,v) \\    D_u g^{(R)} (0,v) 
    \end{array}
\right ]
+ \beta (v)
\left [
    \begin{array}{c}
      D_v  \f F_{0}(v) \\ D_v  g_{0}(v)
    \end{array}
\right ]
=\boldsymbol{0}.
 \end{equation}
\end{definition}
The sign condition \eqref{eq:alpha-sign-condition} on $\alpha^{(L)}$ and $\alpha^{(R)}$  forbids, on
general surfaces, the presence of cusps. However, for the graph of a function it
is obvious.  

For our study it is useful to express $G^1$ continuity as in
Definition \ref{def:G1} since  the coefficients $\alpha^{(L)}$, $\alpha^{(R)}$ and
$\beta$  play an important role.   Since the first two equations
of \eqref{eq:G1-for-Fg-alpha-beta-gamma} are linearly independent,
$\alpha^{(L)}$, $\alpha^{(R)}$ and $\beta$   are uniquely determined,
up to a common multiplicative
factor,  by  $\f F^{(L)} $ and  $\f F^{(R)}$, i.e. from the equation 
\begin{equation}
   \label{eq:G1-for-F-alone-alpha-beta-gamma}
  	\alpha^{(R)} (v)    D_u \f F^{(L)}(0,v) -   \alpha^{(L)}(v)
        D_u  \f F^{(R)}(0,v) + \beta (v) D_v  \f F_{0}(v) 
=\boldsymbol{0}.
\end{equation}
 Precisely, we have the following proposition which  can also be found in \cite{peters1990,peters-PhD,peters-handbook}. 
\begin{proposition}\label{prop:alpha-beta-gamma-are-splines}
  Given any   $\f F^{(L)} $,  $\f F^{(R)}$ in the setting of Section
  \ref{preliminaries},   then
  \eqref{eq:G1-for-F-alone-alpha-beta-gamma} holds if
and only if $\alpha^{(S)} (v)  = \gamma(v) 	\bar \alpha^{(S)} (v)
$,  for $S \in \{L,R\}$,  and  $\beta (v)  = \gamma(v) 	\bar \beta
(v) $, where 
\begin{equation}\label{eq:alphas-wrt-F}
     	\bar   	\alpha^{(S)} (v)=  \det \left[ \begin{array}{ll}
	 D_u \f F^{(S)} (0,v) & D_v  \f F_{0}(v) 
	\end{array}	 \right],
  \end{equation}

 \begin{equation}\label{eq:beta-wrt-F}
   	\bar  	\beta (v) =  	\det \left[ \begin{array}{ll}
	 D_u \f F^{(L)} (0,v)& 	 D_u \f F^{(R)} (0,v) 
	\end{array}
	 \right],
  \end{equation}
and  $\gamma:[0,1]\rightarrow \RR $ is any  scalar function. In addition, $\gamma(v) \neq 0$ if
and only if \eqref{eq:alpha-sign-condition} holds. Moreover, there exist functions $\beta^{(S)}(v)$, for 
$S\in \{L,R\}$, such that 
\begin{equation}
  \label{eq:beta_function_of_alpha}
  \beta (v)= \alpha^{(L)} (v) \beta^{(R)}(v)- \alpha^{(R)} (v)\beta^{(L)}(v).
\end{equation}
\end{proposition}
\begin{proof}
Obviously \eqref{eq:G1-for-F-alone-alpha-beta-gamma}  determines $\alpha^{(L)}(v)$, $\alpha^{(R)}(v)$ and
$\beta(v)$ up to a common factor
$\gamma(v)$ and, since  the two vectors  $D_u \f F^{(L)}
(0,v) $ and $ D_v  \f F_{0}(v)  $ are linearly independent because of 
(\ref{eq:non-singular-F}),  $\alpha^{(L)}(v)$, $\alpha^{(R)}(v)$ and
$\beta(v)$ are uniquely determined up to a common factor
$\gamma(v)$ by \eqref{eq:G1-for-F-alone-alpha-beta-gamma}. 
 We have 
  \begin{equation}\label{eq:1}
    	\det \left[ \begin{array}{lll}
	D_u \f F^{(L)} (0,v) & D_u \f F^{(R)} (0,v)&  D_v  \f F_{0}(v)\\
 	D_u \f F^{(L)} (0,v) \cdot\f e_k& D_u \f F^{(R)} (0,v)  \cdot\f e_k&  D_v  \f F_{0}(v) 
         \cdot\f e_k
	\end{array}
	 \right] = 0,
  \end{equation}
for all $k=1,2$, $ \f e_k$ being the canonical base vectors in
$\RR^2$. Using the Laplace expansion of the above determinant (along the third row) we end up
with 
\begin{displaymath}
  \bar  	\alpha^{(R)} (v) D_u \f F^{(L)} (0,v) -
  \bar   \alpha^{(L)}(v) D_u \f F^{(R)} (0,v) +\bar   \beta (v)  D_v  \f F_{0}(v) =0,
\end{displaymath}
where $\bar  \alpha^{(L)}$, $\bar  \alpha^{(R)}$ and $\bar  \beta$ are as in
\eqref{eq:alphas-wrt-F} and \eqref{eq:beta-wrt-F}. In a similar way \eqref{eq:G1-for-F-det}
yields
\begin{displaymath}
  \bar  	\alpha^{(R)} (v) D_u g^{(L)} (0,v) -
   \bar  \alpha^{(L)}(v) D_u g^{(R)} (0,v) + \bar  \beta (v)  D_v  g_{0}(v) =0.
\end{displaymath}

Setting
\begin{displaymath}
  \bnu (v)  =  \left[ \begin{array}{c}
 D_v  \f F_{0}(v) \cdot\f e_2  \\-  D_v  \f F_{0}(v) \cdot\f e_1
	\end{array}  \right] , \qquad  \btau (v)  =\frac{ D_v  \f F_{0}(v)    }{\| D_v  \f F_{0}(v)  \|^2},
\end{displaymath}
we have
\begin{displaymath}
\det \left[ \begin{array}{lr} \bnu(v) &  \btau (v)	\end{array}
	 \right]= 1\quad  \text{ and } \quad \bnu(v) \cdot  \btau (v) = 0.
\end{displaymath}
 The existence of $\beta^{(S)}(v)$,  
$S\in \{L,R\}$, such that \eqref{eq:beta_function_of_alpha} holds is obvious, because of
\eqref{eq:alpha-sign-condition}. Obviously, $\beta^{(S)}(v)$, for 
$S\in \{L,R\}$, are not unique. A specific choice,  which is of
interest, is the following 
\begin{equation}\label{eq:betaS}
  \beta^{(S)}(v) = \frac{D_u \f F^{(S)} (0,v)  \cdot  D_v  \f F_{0}(v) }{\| D_v  \f F_{0}(v)   \|^2}.
\end{equation}
 Indeed, using the expansion
\begin{displaymath}
  \forall \bv \in \RR^2, \quad  \bv =  \det \left[ \begin{array}{lr} \bv &  \btau (v)	\end{array}
  \right] \bnu (v) -  \det \left[ \begin{array}{lr} \bv &  \bnu (v)	\end{array}
  \right] \btau (v), 
\end{displaymath}
and then
\begin{displaymath}
  \begin{aligned}
  \det \left[ \begin{array}{lr}  \bv^{(L)} & \bv^{(R)} 	\end{array}
  \right]  &=  \det \left[ \begin{array}{lr} \bv^{(L)} &  \bnu (v)	\end{array}
  \right]
  \det \left[ \begin{array}{lr} \bv^{(R)} &  \btau (v)	\end{array}
  \right] \\ & \quad - \det \left[ \begin{array}{lr} \bv^{(L)} &  \btau (v)	\end{array}
  \right]
  \det \left[ \begin{array}{lr} \bv^{(R)} &  \bnu (v)	\end{array}
  \right] , 
  \end{aligned}
\end{displaymath}
gives (\ref{eq:beta_function_of_alpha}) and \eqref{eq:betaS}  
by choosing $  \bv^{(S)} = D_u \f F^{(S)} (0,v) $, for $ S\in \{L,R\}$.
\end{proof}

\begin{remark}
  If \eqref{eq:G1-for-Fg-alpha-beta-gamma} holds, then there exist coefficients $\alpha^{(S)}
  \in \mathcal{S}^{2p-1}_{r-1} ([0,1])$, for $S\in \{L,R\}$, and
  $\beta \in  \mathcal{S}^{2p}_{r} ([0,1])$. Indeed, this follows from
  Proposition \ref{prop:alpha-beta-gamma-are-splines} selecting   $\gamma=1$.
   See also \cite{peters-handbook}. 
\end{remark}

Summarizing,  in the  context of isogeometric methods we consider  $\Omega$ and its
parametrization given.  Then for each interface $\alpha^{(L)}$,
$\alpha^{(R)}$ and $\beta$ are determined from
\eqref{eq:G1-for-F-alone-alpha-beta-gamma} as stated  in Proposition
\ref{prop:alpha-beta-gamma-are-splines}.  It should be observed 
that for planar domains, there always  exist $\alpha^{(L)}$,
$\alpha^{(R)}$ and $\beta$ fulfilling
\eqref{eq:G1-for-F-alone-alpha-beta-gamma}. This is not the case for
surfaces, see Section \ref{surfaces}. Then,  the $C^1$ continuity
of isogeometric functions  is equivalent to the last equation in
\eqref{eq:G1-for-Fg-alpha-beta-gamma}, that is 
\begin{equation}\label{eq:G1-condition-for-g}
    \alpha^{(R)} (v) D_u g^{(L)} (0,v) -
   \alpha^{(L)}(v) D_u g^{(R)} (0,v) + \beta (v)  D_v  g_{0}(v) =0
\end{equation}
for all $v \in [0,1]$. 

We end this section by  a statement of  the equivalence between $C^1$
continuity of the isogeometric function and $G^1$ continuity of its graph 
parametrization. It 
is formalized and presented in its most general form for arbitrary continuity and dimension
in \cite{groisser2015matched}. The use of $G^1$ continuous functions over unstructured mesh partitions is well known in the isogeometric
community, see e.g. \cite{kiendl-bazilevs-hsu-wuechner-bletzinger-10,scottPhD,scott2013isogeometric,kapl-vitrih-juttler-birner-15,mourrain2015geometrically,bercovier2014smooth}.   We give a 
detailed proof of the statement here, in the framework of Proposition \ref{prop:alpha-beta-gamma-are-splines}, 
since this will serve for the next steps of Section \ref{two-patch}. 

\begin{proposition}\label{teo:C1=G1-graph-parametrization}
An isogeometric function $\phi \in \mathcal{V}$ belongs to $
  \mathcal{V}^1$ if and only if 
its graph $\Sigma$ is $G^1$ continuous  on each interface $\Sigma^{(i)} \cap
  \Sigma^{(j)} $.
\end{proposition}
\begin{proof}
   Consider a graph interface $\Sigma^{(i)} \cap \Sigma^{(j)} $  and
  the corresponding  $\Gamma = \Gamma^{(i,j)} =  \Omega^{(i)}
  \cap  \Omega^{(j)}$.   Let  $  \alpha^{(S)}(v)$, $  \beta^{(S)}(v)$,
  for $S\in\{L,R\}$  such that
  \eqref{eq:G1-for-F-alone-alpha-beta-gamma} holds.  Define the vector
  $  \bn^{(S)} $  on $\Gamma$ such that 
  \begin{equation}\label{eq:bn}
  	\bn^{(S)} \circ \f F_0(v) =
	\left[ \begin{array}{cc}
	 	D_u \f F^{(S)} (0,v)   &   D_v  \f F_{0}(v) 
	\end{array}  \right] \left[ \begin{array}{c}
	 1    \\
	-\beta^{(S)}(v)
	\end{array}  \right] \frac{1} { \alpha^{(S)}(v)}.
  \end{equation}
The vector   $  \bn^{(S)} $ is transversal to  $\Gamma$, i.e. linear
independent to $ D_v  \f F_{0}(v) $,  since
\begin{displaymath}
  \det \left[ \begin{array}{lr}  \bn^{(S)} \circ \f F_0(v)  & D_v  \f F_{0}(v) 	\end{array}
  \right]    = \frac{1} { \alpha^{(S)}(v)} \det
  \left[ \begin{array}{lr}  	D_u \f F^{(S)} (0,v)    & D_v  \f F_{0}(v) 	\end{array}
  \right]   = \frac{1} { \gamma(v)} \neq 0.
  \end{displaymath}
We have $\bn^{(L)}=\bn^{(R)}=\bn$. Indeed, by using
\eqref{eq:G1-for-F-alone-alpha-beta-gamma} and
\eqref{eq:beta_function_of_alpha}, we get
\begin{displaymath}
  \begin{aligned}
    & \alpha^{(L)}(v) \alpha^{(R)}(v)	\bn^{(L)} \circ \f F_0(v) -
     \alpha^{(L)}(v) \alpha^{(R)}(v)	\bn^{(R)} \circ \f F_0(v)  \\
     & \quad = \alpha^{(R)}(v) 	D_u \f F^{(L)} (0,v)  -
     \alpha^{(L)}(v) 	D_u \f F^{(R)} (0,v) + (\alpha^{(L)}(v)
     \beta^{(R)}(v) -  \alpha^{(R)}(v) \beta^{(L)}(v))     D_v  \f F_{0}(v)  \\
     & \quad = \boldsymbol 0.
  \end{aligned}
\end{displaymath}
Futhrermore, $\bn$   is not in
general unitary and it is  continuous, since $\f F_{0}$ is at least
$C^1$ from \eqref{eq:r-geq-1}.
For $S\in\{L,R\}$, let $\phi = g^{(S)} \circ \f [F^{(S)}]^{-1}$ on $\Omega^{(S)}$, then
the derivative from side $S$ in direction $\bn$ fulfills 
\begin{equation}
  \label{eq:phi-grad}
  \begin{aligned}
      \nabla^{(S)} \phi (\bx) \cdot  \bn (\bx ) & = [ 	D_u g^{(S)} (0,v)  \,  D_v  g_{0}(v)]  \left[ \begin{array}{c}
	 	D_u \f F^{(S)} (0,v)   \,   D_v  \f F_{0}(v) 
	\end{array}  \right]^{-1} \cdot  \bn \circ \f F_0(v) ,
  \end{aligned}
\end{equation}
for all $(\bx) \in \Gamma $, 
where here and in what follows we implicitly assume the relation $\f F_{0}
  (v) =(\bx) $.  
 
We obtain directly from the definition of $\bn^{(S)}$ that 
 \begin{equation}
  \label{eq:Finv-grad}
   \left[ \begin{array}{c}
	 	D_u \f F^{(S)} (0,v)   \,   D_v  \f F_{0}(v) 
	\end{array}  \right]^{-1}  \cdot \bn \circ \f F_0(v)  =\frac{1}{ \alpha^{(S)}(v)} \left[ \begin{array}{c}
	 1    \\
	-\beta^{(S)}(v)
	\end{array}  \right] .
\end{equation}

Substituting \eqref{eq:Finv-grad} back to \eqref{eq:phi-grad} we then obtain 
\begin{equation}
  \label{eq:phi-normal-derivative-S}
  \begin{aligned}
      \nabla^{(S)} \phi (\bx) \cdot  \bn (\bx )
    & =\frac{D_u g^{(S)} (0,v)   -\beta^{(S)}(v)   D_v  g_{0}(v) }{ \alpha^{(S)}(v)}
  \end{aligned}
\end{equation}
for $S \in \{R,L\}$. Therefore  $   \nabla^{(L)} \phi (\bx) \cdot  \bn
(\bx ) =    \nabla^{(R)} \phi (\bx) \cdot  \bn (\bx )$  if and only if (by \eqref{eq:phi-normal-derivative-S})
\begin{displaymath}
  \frac{D_u g^{(L)} (0,v)   -\beta^{(L)}(v)   D_v  g_{0}(v) 
    }{ \alpha^{(L)}(v)} = \frac{D_u g^{(R)} (0,v)   -\beta^{(R)}(v)   D_v  g_{0}(v) }{ \alpha^{(R)}(v)}.
\end{displaymath}
That, after multiplying both sides by $ \alpha^{(L)}(v)
\alpha^{(R)}(v)  $ and  using 
\eqref{eq:beta_function_of_alpha}, is equivalent to \eqref{eq:G1-condition-for-g}. 
\end{proof}

\begin{remark}\label{rem:C1-isoparametric-on-G1-surfaces}
 As a consequence of Proposition~\ref{teo:C1=G1-graph-parametrization},
 $C^1$ isogeometric spaces over $C^0$ planar multi-patch
 parametrizations fit in the  isoparametric framework. 
\end{remark}

\section{Analysis-suitable $G^1$ parametrizations}
\label{AS-G1-parametrization}

At each interface  $\Gamma = \Gamma^{(i,j)} =  \Omega^{(i)} \cap 
\Omega^{(j)}$, given  any regular and orientation preserving  $\f F^{(L)} $,  $\f F^{(R)}$,  as in 
  \eqref{eq:F-left-and-right},  there
  exist coefficients $\alpha^{(L)}$,
$\alpha^{(R)}$ and $\beta$, with    $\alpha^{(L)}  (v) \alpha^{(R)} (v) > 0$,
 such that \eqref{eq:G1-for-F-alone-alpha-beta-gamma}  holds. 
 Then \eqref{eq:G1-condition-for-g} expresses $C^1$ continuity  of
 isogeometric functions  in terms of $\alpha^{(L)}$,
$\alpha^{(R)}$ and $\beta$. Optimal approximation properties  of the
isogeometric space on $\Omega$ hold under  restrictions on $\alpha^{(L)}$,
$\alpha^{(R)}$ and $\beta$, i.e.  on the geometry
parametrization. This leads to the definition below.

\begin{definition}[Analysis-suitable $G^1$-continuity] \label{def:AS-G1}   $\f F^{(L)} $ and $\f F^{(R)}$ are \emph{analysis-suitable $G^1$-continuous}  at the interface $ \Gamma$ (in short,
  AS $G^1(\Gamma)$ or  AS $G^1$)  if there exist  $\alpha^{(L)} , \alpha^{(R)}, \beta^{(L)}, \beta^{(R)}\in \mathcal{P}^{1}
  ([0,1]) $  such that  \eqref{eq:G1-for-F-alone-alpha-beta-gamma} and
  (\ref{eq:beta_function_of_alpha}) hold, that is for all $v \in [0,1]$
  \begin{displaymath}
    	\alpha^{(R)} (v)    D_u \f F^{(L)}(0,v) -   \alpha^{(L)}(v)
        D_u  \f F^{(R)}(0,v) = (\alpha^{(R)} (v) \beta^{(L)}(v)- \alpha^{(L)} (v)\beta^{(R)}(v)) D_v  \f F_{0}(v).
  \end{displaymath}
\end{definition}
The degrees of the functions $\alpha^{(L)}$, $\alpha^{(R)}$ and $\beta$ were also studied in the 
context of $G^1$ interpolation of a mesh of curves in \cite{peters1990,peters1991smooth}. 
There, the same degrees where derived for interpolations using cubic patches.

\begin{remark}\label{rem:C1-isoparametric-on-ASG1-surfaces}
 As in Remark \ref{rem:C1-isoparametric-on-G1-surfaces}, 
 we observe that $C^1$ isogeometric spaces over  AS $G^1$  multi-patch
 parametrizations fit in the  isoparametric framework. 
\end{remark}

The class of  \emph{AS} $G^1$ parametrizations contains the
bilinear ones and more
\begin{proposition}\label{prop:bilinear}
 Any  $\f F^{(L)} \in \mathcal{P}^{1} (\Omega^{(L)}) \times \mathcal{P}^{1} (\Omega^{(L)})$ and $\f
 F^{(R)}\in \mathcal{P}^{1} (\Omega^{(R)}) \times \mathcal{P}^{1} (\Omega^{(R)})$ are  AS $G^1$-continuous
 at  $ \Gamma$. Moreover, for any $\alpha^{(L)},\alpha^{(R)} \in \mathcal{P}^{1} 
  ([0,1]) $  strictly
 positive and  $\beta^{(L)}, \beta^{(R)}\in \mathcal{P}^{1}
  ([0,1]) $    there exist   $\f F^{(L)} \in \mathcal{P}^{1} (\Omega^{(L)})$ and $\f
 F^{(R)}\in \mathcal{P}^{1} (\Omega^{(R)})$ fulfilling \eqref{eq:G1-for-F-alone-alpha-beta-gamma}.
\end{proposition}
\begin{proof}
  The statement follows directly from Proposition
  \ref{prop:alpha-beta-gamma-are-splines}  and \eqref{eq:betaS}.
\end{proof}

\begin{remark}\label{rem:gamma-for-ASG1}
    The class of \emph{AS} $G^1$ parametrizations is
wider than only bilinear. We will show some examples later in Section
\ref{numerical_tests}  where $\f F^{(L)}$ and $\f F^{(R)}$ are higher
order polynomials at $\{0\}\times [0,1]$.
\end{remark}

\section{Two-patch geometry}
\label{two-patch}

In this section we analyze the  two-patch geometry.  This is the simplest
geometric configuration that allows us to focus on the $C^1$-continuity
constraints that are associated with each patch interface.
For the space of isogeometric $C^0$ functions, that is $\mathcal{V}^0$
defined in \eqref{eq:V}-\eqref{eq:V0}, optimal convergence under
$h$-refinement  is known since the results in \cite{bazilevs2006isogeometric,da2012anisotropic}  apply directly.
$C^1$-continuity constrains  traces 
and transversal  derivatives at $\Gamma$ of  functions $ \phi \in
\mathcal{V}^1 $. Therefore the approximation
properties of the space  $ \mathcal{V}^1 $  follow from  the ones  of traces of
functions and 
transversal derivatives of functions of  $ \mathcal{V}^1 $ 
 at  $\Gamma$.  
Let  $  \alpha^{(S)}(v)$, $  \beta^{(S)}(v)$, for $S\in\{L,R\}$  such that
  \eqref{eq:G1-for-F-alone-alpha-beta-gamma} holds.  We define the
  transversal  vector
  $  \bn$ to  $\Gamma$  as in \eqref{eq:bn} and introduce the space of  traces and  transversal directional 
 derivatives on $\Gamma$  
 \begin{equation} \label{eq:V1-Gamma}
  \mathcal{V}^1_{\Gamma} = 
   \left \{  \Gamma \ni (x,y) \mapsto [\phi(x,y) , \nabla \phi (x,y) 
   \cdot \bn (x,y) ] , \text{ such that }  \phi \in \mathcal{V}^1 \right \}
 \end{equation}
and its pullback \begin{equation} \label{eq:V1-Gamma-pullback}
 \widehat{ \mathcal{V}^1_{\Gamma}} = 
   \left \{  [\phi, \nabla \phi
   \cdot \bn ] \circ \f F_0, \text{ such that }  \phi \in \mathcal{V}^1 \right \}.
 \end{equation}
 With this choice we have that 
$ [\phi, \nabla \phi
   \cdot \bn ] \circ \f F_0 = [g_0,g_1]$ is equivalent to 
\begin{equation}
  \label{eq:equivalence-phi-g}
      g^{(S)} (u,v) = g_0(v) + ( \beta^{(S)}(v) g_0'(v) + \alpha^{(S)}(v)g_1(v)) \, u
      + O(u^2) ,
\end{equation}
thanks to 
\eqref{eq:phi-normal-derivative-S}.
\begin{remark}
   The   transversal  vector
  $  \bn$  depends on $\alpha^{(L)}$,  $\alpha^{(R)}$,
  $\beta^{(L)}$, $\beta^{(L)}$, and so do $  \mathcal{V}^1_{\Gamma} $
  and $ \widehat{ \mathcal{V}^1_{\Gamma}} $. The function trace in $\widehat{ \mathcal{V}^1_{\Gamma}}$
fulfills $\phi \circ \f F_0 \in \mathcal{S}^p_r$, this is not true for
the   transversal derivative $ \nabla \phi
   \cdot \bn$  which in general is a rational function of some degree
   higher than $p$ (a similar situation is analyzed in
   \cite{Takacs2014}). 
\end{remark}

\subsection{AS $G^1$-continuous two-patch geometry}
\label{two-patch-AS}
The next result  gives the key properties 
of the space 
$\widehat{ \mathcal{V}^1_{\Gamma} }$ (defined in \eqref{eq:V1-Gamma-pullback})
in the case of AS  $G^1$ parametrizations.
\begin{theorem}\label{thm:trace-for-AS-G1}
Let $\Omega= \Omega^{(L)}\cup  \Omega^{(R)}$, and let   $\f F^{(L)} $ and
$\f F^{(R)}$ be  AS  $G^1$ at the
interface $ \Gamma = \partial  \Omega^{(L)}\cup  \partial \Omega^{(R)}$.
 Then 
\begin{equation}
  \label{eq:trace-and-normal-derivative-for-ASG1}
 \S^p_{r+1}([0,1]) \times \S^{p-1}_{r}  ([0,1])\subset  \widehat{ \mathcal{V}^1_{\Gamma}}.
\end{equation}
\end{theorem}
\begin{proof}
For linear $\alpha^{(S)} $ and
$\beta^{(S)} $  ($S \in \{L,R\}$) and $[g_0,g_1] \in \S^p_{r+1}([0,1])
\times \S^{p-1}_{r}  ([0,1])$, we have  
\begin{equation}
  \label{eq:equivalence-phi-g-2}
      g^{(S)} (u,v) = g_0(v) + ( \beta^{(S)}(v) g_0'(v) + \alpha^{(S)}(v) g_1(v)) \, u
       \in \S^p_{r}(\hat \Omega^{(S)} ) . 
\end{equation}
Then the statement is a direct consequence of
\eqref{eq:equivalence-phi-g}. 
\end{proof}

Theorem \ref{thm:trace-for-AS-G1}
guarantees that the trace space for the function value $\{ \phi \circ \f F_0 : \phi
\in  \mathcal{V}^1\}$ includes all
splines of degree $p$ and regularity at least $r+1$, and independently
the trace space for the transversal derivatives of the function  $\{(\nabla  \phi
   \cdot \bn )\circ \f F_0 :\phi
\in  \mathcal{V}^1\}$  includes all
splines of degree $p-1$ and regularity at least $r$.  This
suggests that  $  \mathcal{V}^1_{\Gamma}  $  enjoys  optimal
approximation order, and consequently for the whole space $  \mathcal{V}^1$ when  $r< p-1$.

However, if $r= p-1$ and the parametrization is not trivial, the space  $  \mathcal{V}^1_{\Gamma}  $ suffers
of $C^1$ locking, that is, $h$-refinement does not improve the  approximation
properties  for $\phi $ and $\nabla \phi
   \cdot \bn $  independently.  The following theorem gives some
   understanding of this phenomenon. 

\begin{theorem}\label{thm:C1-locking-for-AS-G1}
Let $\Omega= \Omega^{(L)}\cup  \Omega^{(R)}$,   $\f F^{(L)} $,
$\f F^{(R)}$ be  AS  $G^1$ at the
interface $ \Gamma =   \Omega^{(L)}\cap \Omega^{(R)}$,   and $r= p-1$. Furthermore, 
let $\mathcal{G}_0 $ and $   \mathcal{G}_1$ be two spaces such that
\begin{equation}\label{eq:tensor-product-subspace}
  \mathcal{G}_0 \times  \mathcal{G}_1 \subseteq \widehat{ \mathcal{V}^1_{\Gamma}}.
\end{equation}
If either  $\beta^{(L)} \neq  0 $ or  $\beta^{(R)} \neq  0 $ then  the
dimension of  $\mathcal{G}_0$  is independent of $h$.
If   $\alpha^{(L)}$ and $\alpha^{(R)}$ are linearly independent,  then  the
dimension of  $\mathcal{G}_1$  is independent of $h$.
\end{theorem}
\begin{proof}
  Let $  [g_0,0] = [\phi, \nabla \phi
   \cdot \bn ] \circ \f F_0  $   and assume  that $\beta^{(L)} $ is a  linear function not
 identically zero.   Furthermore assume that if there exists a $v_0 \in [0,1]$, such that
$\beta^{(L)} (v_0)= 0$, then from the assumption $\beta^{(R)} (v_0)\neq 0$. Using
\eqref{eq:equivalence-phi-g},  i.e., 
   \begin{displaymath}
      g^{(S)} (u,v) = g_0(v) + \beta^{(S)}(v) g_0'(v) \, u
      + O(u^2) ,
   \end{displaymath}
and using $ g^{(S)} (u,v)  \in \S^p_{p-1}(\hat \Omega^{(S)}
)$  then  $ g_0 $ is a spline in $\S^p_{p-1}([0,1])$ and also $ g'_0 $ is
 $C^{p-1}$, therefore  $ g_0 $ is in fact a degree $p$ global polynomial. The same
 conclusion holds when  there exists a $v_0 \in [0,1]$ such that
$\beta^{(L)} (v_0)=\beta^{(R)} (v_0) = 0$ but $v_0 $ is not a knot of
the spline space $\S^p_{p-1}([0,1])$ under consideration. If, instead,
$\beta^{(L)} (v_0)=\beta^{(R)} (v_0) = 0$ and  $v_0 $ is a knot of
the spline space $S^p_{p-1}([0,1])$, then  $ g_0 $ has in
general  $p-1$ continuous derivatives at $v_0$. In conclusion if $g_0
\in \mathcal{G}_0$  then it has to be piecewise polynomial of degree
$p$ on at most two intervals, and is in $C^{p-1}([0,1])$ globally. Hence the dimension
of  $\mathcal{G}_0$ is at most  $p+2$.

Take now $  [0,g_1] = [\phi, \nabla \phi
   \cdot \bn ] \circ \f F_0  $  and assume  $\alpha^{(L)}$ and
   $\alpha^{(R)}$ are linearly independent, that is, their linear combination gives the whole space $\mathcal{P}^1([0,1])$.
Now \eqref{eq:equivalence-phi-g}, i.e., 
   \begin{displaymath}
      g^{(S)} (u,v) =  \alpha^{(S)}(v) g_1(v) \, u
      + O(u^2),
   \end{displaymath}
 after taking the derivative in the
variable $u$ and evaluating at $u=0$ gives 
   \begin{displaymath}
       \frac{\partial}{\partial u}      g^{(S)} (0,v) =  \alpha^{(S)}(v) g_1(v), 
   \end{displaymath}
and, by suitable linear combination for $S=L$ and $S=R$ yields
\begin{displaymath}
  C_L      \frac{\partial}{\partial u}  g^{(L)} (0,v) + C_R
  \frac{\partial}{\partial u}  g^{(R)} (0,v) = g_1(v).
\end{displaymath}
Since $ g^{(S)} (u,v)  \in \S^p_{p-1}(\hat \Omega^{(S)}
)$, $g_1 \in \S^p_{p-1}([0,1])$. Similarly
\begin{displaymath}
  C_L^\prime \frac{\partial}{\partial u}  g^{(L)} (0,v) + C_R^\prime
  \frac{\partial}{\partial u}  g^{(R)} (0,v) = v g_1(v),
\end{displaymath}
which means $vg_1 \in \S^{p}_{p-1}([0,1])$. in conclusion $g_1 \in \S^{p-1}_{p-1}([0,1])$, therefore  the dimension
of  $\mathcal{G}_1$ is at most  $p$.
\end{proof}
 When  $\beta^{(L)} =
   \beta^{(R)} = 0 $  and  for all $v \in [0,1]$ we have $ \alpha^{(L)} (v) = C
   \alpha^{(R)} (v) $ then 
   \begin{displaymath}
     D_u \f F^{(L)}(0,v) =  C            D_u  \f F^{(R)}(0,v) ,
   \end{displaymath}
that is,  the parametrization is trivially $G^1$ in the sense that it can be made $C^1$ by scaling the variable $u$ by a
multiplicative factor in (one of) the two patches. If we are in such a case, as for $C^1$
parametrizations, one can easily conclude $\S^p_{r}([0,1]) \times \S^{p}_{r}
([0,1]) =   \widehat{ \mathcal{V}^1_{\Gamma}}$. If we are not in such a
special configuration, then  Theorem \ref{thm:C1-locking-for-AS-G1}
states that we can \emph{independently} approximate the trace and
transversal derivative of a function by isogeometric $C^1$ functions, 
but for $h$-refinement the approximation error does not converge to zero.

\subsection{General two-patch geometry}
\label{two-patch-GENERAL}

In this section we study  the  approximation property of
$\mathcal{V}^1$,  for   two-patch geometry
parametrizations  beyond analysis-suitable $G^1$-continuity.  
 We consider a specific case where  $\f F^{(L)} $ is the
identity mapping while $\f F^{(R)}  \in  [\S^{p}_{r}(\hat \Omega^{(R)})]^2 $, and select 

\begin{equation}
  \label{eq:parameters-for-FL-identity}
  \begin{aligned}
     \alpha^{(L)}(v) & = 1\\
     \alpha^{(R)}(v) & = D_u \f F^{(R)} (0,v) \cdot\f e_1 \in \S^{p}_{r}([0,1]), \\ 
     \beta^{(L)} (v) & = 0\\ 
     \beta^{(L)} (v) & = D_u \f F^{(R)} (0,v) \cdot\f e_2 \in \S^{p}_{r}([0,1]).
  \end{aligned}
\end{equation}
Then $\bn = \f e_1  $ and  
\begin{equation}\label{eq:V1Gamma-nonASG1}
   \widehat{
   \mathcal{V}^1_{\Gamma}}  \equiv\mathcal{V}^1_{\Gamma}    = \left \{ \left
       [\phi, \frac{\partial}{\partial x}   \phi \right ], \text{ such that }  \phi \in \mathcal{V}^1 \right \}.
\end{equation}
 Equation \eqref{eq:equivalence-phi-g} simplifies to 
\begin{equation}
  \label{eq:equivalence-phi-g-Left-identity}
  \begin{aligned}&\left [\phi, \frac{\partial \phi}{\partial x}\right ] \circ \f F_0 = [g_0,g_1] 
    \\ & \quad \Leftrightarrow 
    \left \{\begin{aligned}
      g^{(L)} (u,v) & = g_0(v) +   g_1(v) \, u
      + O(u^2) ,\\  
      g^{(R)} (u,v) &= g_0(v) + ( \beta^{(R)}(v) g_0'(v) + \alpha^{(R)}(v) g_1(v)) \, u
      + O(u^2).
    \end{aligned} \right . 
  \end{aligned}
\end{equation}
 The following statement gives a full characterization of the space ${
   \mathcal{V}^1_{\Gamma}}  $. We use 
$\S^{-1}_{r}([0,1]) $ to indicate the null space $\emptyset$, and we recall that $p_{\alpha}=0$ is not allowed.

\begin{theorem}\label{thm:trace-for-general-parametrization}
Let $\Omega= \Omega^{(L)}\cup  \Omega^{(R)}$,   $\f F^{(L)} $
be the identity and $  \alpha^{(R)}$,  $  \alpha^{(L)}$, $\beta^{(R)}  $,
$\beta^{(L)}  $ as in \eqref{eq:parameters-for-FL-identity}. Assume  $   \alpha^{(R)} \in  \S^{p_{\alpha}}_{r}([0,1])
$, with $\alpha^{(R)} \not \in  \S^{p_{\alpha}-1}_{r}([0,1]) $.  

If $\beta^{(R)} = 0  $, then
\begin{equation}
  \label{eq:nonASG1-beta-zero}
  {   \mathcal{V}^1_{\Gamma}} = \S^{p}_{r}([0,1]) \times
 \S^{p-p_{\alpha}}_{r} ([0,1]).
\end{equation}

 If instead  $   \beta^{(R)} \in  \S^{p_{\beta}}_{r}([0,1])  $ with $   \beta^{(R)}
\not \in  \S^{p_{\beta}-1}_{r}([0,1])  $, 
 then
\begin{equation}
  \label{eq:nonASG1-beta-nonzero}
  {   \mathcal{V}^1_{\Gamma}} = \S^{\min\{p,p-p_{\beta}+1\}}_{r+1}([0,1])
 \times \S^{p-p_{\alpha}}_{r} ([0,1]).
\end{equation}
\end{theorem}
\begin{proof} 
Since $\f F^{(L)} $ is the identity  we have  $
\mathcal{V}^1_{\Gamma} = \widehat{
   \mathcal{V}^1_{\Gamma}} \subset \S^{p}_{r}([0,1]) \times
 \S^{p}_{r} ([0,1])$. 

Consider first  the case  $\beta^{(R)} = 0  $. It is easy to see that 
\[ 
\widehat{ \mathcal{V}^1_{\Gamma}} \supset \S^{p}_{r}([0,1]) \times \S^{p-p_{\alpha}}_{r} ([0,1]).
\] 
Indeed, given any $[g_0,g_1]  \in \S^{p}_{r}([0,1]) \times
 \S^{p-p_{\alpha}}_{r} ([0,1])$ we can find  $ g^{(S)} (u,v)  \in \S^p_{r}(\hat
\Omega^{(S)} )$ such that \eqref{eq:equivalence-phi-g-Left-identity}
holds. In particular, take 
\[
	g^{(R)} (u,v) = g_0(v) + \alpha^{(R)}(v) g_1(v)\, u.
\]  
Moreover,  $\S^{p}_{r}([0,1]) \times \S^{p-p_{\alpha}}_{r} ([0,1])$ 
is equal to the space  $ \widehat{\mathcal{V}^1_{\Gamma}} $ 
due to the second condition in \eqref{eq:equivalence-phi-g-Left-identity}, which is 
\[
	g^{(R)} (u,v) = g_0(v) + \alpha^{(R)}(v) g_1(v)\, u  + O(u^2) 
\]
for $\beta^{(R)}  = 0  $, since $ g^{(R)} (u,v) \in \S^p_{r}(\hat \Omega^{(R)} )$ forbids  $
g_1$ to be of a polynomial degree higher than $p-p_{\alpha}$.

The second case,   $   \beta^{(R)} \in  \S^{p_{\beta}}_{r}([0,1])  $ with $   \beta^{(R)}
\not \in  \S^{p_{\beta}-1}_{r}([0,1])  $, is similar.  Again, we have $ \widehat{
   \mathcal{V}^1_{\Gamma}} \supset  \S^{\min\{p,p-p_{\beta}+1\}}_{r+1}([0,1])
 \times \S^{p-p_{\alpha}}_{r} ([0,1])$. Indeed, given any 
 \[
 	[g_0,g_1]  \in  \S^{\min\{p,p-p_{\beta}+1\}}_{r+1}([0,1]) \times \S^{p-p_{\alpha}}_{r} ([0,1])
\]
we can find  $ g^{(S)} (u,v)  \in \S^p_{r}(\hat \Omega^{(S)} )$ 
such that \eqref{eq:equivalence-phi-g-Left-identity} holds. This time, take 
\[
	g^{(R)} (u,v) = g_0(v) + ( \beta^{(R)}(v) g_0'(v) +
     \alpha^{(R)}(v) g_1(v)) \, u.
\]
As before  $ \S^{\min\{p,p-p_{\beta}+1\}}_{r+1}([0,1]) \times \S^{p-p_{\alpha}}_{r} ([0,1])$ is equal to the space 
  $ \widehat{ \mathcal{V}^1_{\Gamma}} $ due to the second condition in \eqref{eq:equivalence-phi-g-Left-identity}, that is  
\[ 
	\S^p_{r}(\hat \Omega^{(R)} )  \ni g^{(R)} (u,v) = g_0(v) + ( \beta^{(R)}(v) g_0'(v) + \alpha^{(R)}(v) g_1(v)) \, u  + O(u^2).
\]
This concludes the proof.
\end{proof}
\begin{remark}
 In the setting of this section, $\f F^{(L)}  $ being the identity and
 $ \f F^{(R)} \in \mathcal{S}^{p}_{r} (\widehat\Omega^{(R)})\times
 \mathcal{S}^{p}_{r} (\widehat\Omega^{(R)})$, we always have $ p_\alpha
 \leq p$, $ p_\beta  \leq p$, which follows from
   \eqref{eq:parameters-for-FL-identity}. Then Theorem~\ref{thm:trace-for-general-parametrization} guarantees
 \begin{displaymath}
   \mathcal{V}^1_{\Gamma}  =  \widehat{
   \mathcal{V}^1_{\Gamma}}  \supset \S^{1}_{p}([0,1]) \times
   \S^{0}_{p-1}([0,1])   = \mathcal{P}^1 \times \mathcal{P}^0.
 \end{displaymath}
This is in accordance with the well known fact that  isoparametric
functions reproduce all linear polynomials.
\end{remark}
 We remark that any   restriction on
$\mathcal{V}^1_{\Gamma} $ as stated in Theorem
\ref{thm:trace-for-general-parametrization} affects the approximation
properties of  $ \V^{1}|_{\Omega^{(L)}} \subset \S^{p}_{r}(
\Omega^{(L)})$ and consequently of  $ \V^{1}$ itself. For the sake of
convenience, this is summarized in
the corollary below. 
\begin{corollary}\label{cor:convergence-general-parametrization}
Optimal order of convergence for $h$-refinement can not be achieved  if
 $ \deg  \alpha^{(R)} > 1 $ or  $ \deg  \beta^{(R)} >1$. In
 particular  $C^1$ locking, i.e.  no convergence, occurs  for  $ \deg  \alpha^{(R)} \geq p-r
 $ or  $ \deg  \beta^{(R)} \geq p-r$.
\end{corollary}
 Some examples will be considered
 and further analyzed in Section~\ref{numerical_tests}.

\section{$C^1$ isogeometric spaces on surfaces}
\label{surfaces}

The result of the previous sections can be extended to a a multi-patch
surface $	\Omega $
\begin{equation*}
	\Omega = \Omega^{(1)} \cup \ldots \cup \Omega^{(N)} \subset \RR^3.
\end{equation*}
We can set up an isogeometric function space over the surface
  \begin{equation}\label{eq:Vsurface}
    \mathcal{V} = \left \{ \phi:\Omega \rightarrow
      \mathbb{R}\text{ such that }  \phi \circ \f F^{(i)} \in
  \mathcal{S}^{p}_{r} (\hat \Omega) , i=1,\ldots,N \right \},
  \end{equation}
and again $ \mathcal{V}^{0}=\mathcal{V}\cap C^0(\Omega)$ and
$\mathcal{V}^{1}=\mathcal{V}\cap C^1(\Omega)$. 
As for the planar case, $	\f F^{(i)}:\widehat\Omega\rightarrow \Omega^{(i)}$, though $ \Omega^{(i)} $ is
now a surface patch. More important,  for the function space $C^1(\Omega)$ to be well defined, the
surface $\Omega$ itself needs to be $C^1$, i.e.,  the surface 
needs to have a well defined tangent plane in every point. 
For simplicity,  we focus on a single interface, that is a two-patch geometry, where each patch is parameterized via 
\begin{equation}
  \begin{aligned}
    	\f F^{(L)}&: [-1,0]\times[0,1]=
        \widehat\Omega^{(L)}\rightarrow \Omega^{(L)} \subset \RR^3, \\
	\f F^{(R)}&: [0,1]\times[0,1]=
        \widehat\Omega^{(R)}\rightarrow \Omega^{(R)}\subset \RR^3
  \end{aligned}
\end{equation}
with $\f F^{(S)}\in (\mathcal{S}^{p}_{r} (\widehat\Omega^{(S)}))^3$ for $S\in\{L,R\}$.
We ask  the parameterization to be $G^1$, i.e., there exist $\alpha^{(L)}  : [0,1]
 \rightarrow \RR $, $ \alpha^{(R)} : [0,1] \rightarrow \RR $ and $\beta  : [0,1]
 \rightarrow \RR $ such that $\forall v \in [0,1]$,  $\alpha^{(L)}
 (v) \alpha^{(R)} (v) > 0$ and \begin{equation}
   \label{eq:G1-for-F-surf-alpha-beta-gamma}
  	\alpha^{(R)} (v)    D_u \f F^{(L)}(0,v) -   \alpha^{(L)}(v)
        D_u  \f F^{(R)}(0,v) + \beta (v) D_v  \f F_{0}(v) 
=\boldsymbol{0}.
 \end{equation}
Then, the surface gradient of a function $\phi:\Omega \rightarrow \RR$ can be computed as follows. First, the surface $\Omega$ is extended to $\Omega_\epsilon\subset\RR^3$ by defining 
\begin{equation}
  \f G^{(S)} :  \widehat\Omega^{(S)} \times \left]-\epsilon,\epsilon\right[ \rightarrow \Omega^{(S)}_\epsilon
\end{equation}
with 
\begin{equation}
  \f G^{(S)}(u,v,w) = \f F^{(S)}(u,v) + w \f N^{(S)}(u,v),
\end{equation}
where $\f N^{(S)}$ is the unit normal vector of $\Omega^{(S)}$. Now, we define the extension $\Phi$ of the function $\phi$ via $\hat\Phi^{(S)}(u,v,w) = \hat\phi^{(S)}(u,v)$ for $S\in\{L,R\}$. The surface gradient of $\phi$ is then given as the three-dimensional gradient of the extension $\Phi$ in $\Omega_\epsilon$ restricted to the surface $\Omega$, i.e. 
\begin{equation}
  \nabla_\Omega \phi = \nabla \Phi |_{\Omega} ,
\end{equation}
where
\begin{equation}
  \nabla \Phi \circ \f G^{(S)} = \nabla \hat\Phi^{(S)} \cdot (\nabla \f G^{(S)})^{-1}\; \mbox{ on }\widehat\Omega^{(S)}.
\end{equation}
By construction, the gradient is tangential to the surface.

Proposition \ref{teo:C1=G1-graph-parametrization} can be generalized to surface
domains. For the sake of simplicity we consider a simplified
case, by assuming that $\f F^{(S)}$ can be projected onto the $(x,y)$-plane without self intersections.  Note that this is not a limitation of the concept but facilitates the following propositions. Then we have 
\begin{equation}\label{eq:projectable-surface-parametrization}
  \f F^{(S)} = (F^{(S)}_1,F^{(S)}_2,F^{(S)}_3)^T=  (\f P^{(S)} , f^{(S)}_3 \circ \f P^{(S)})^T
\end{equation}
where $\f P^{(S)} = (F^{(S)}_1,F^{(S)}_2)^T$ is the planar
parameterization of the projected surface and $f_3$ is an  (isogeometric)
function from the planar projection $\bar \Omega^{(S)} = \f
P^{(S)}(\widehat\Omega^{(S)}) $ to  $\RR$. Then the following
result is a direct corollary of Proposition
\ref{teo:C1=G1-graph-parametrization},  see also \cite{groisser2015matched}. 
\begin{proposition}
  An isogeometric function $\phi \in \mathcal{V}$ as in \eqref{eq:Vsurface} belongs to $
  \mathcal{V}^1$ if and only if its graph $\Sigma$ as a 2-dimensional surface in $\RR^4$ is $G^1$-continuous at each interface $\Sigma^{(i)} \cap
  \Sigma^{(j)} $.
\end{proposition}
The definition for a $G^1$-continuous  graph of an isogeometric function
$\phi$  is formally the
same as in  Definition \ref{def:G1}, however now  the graph of $\phi:
\Omega \rightarrow \mathbb{R}$ is a two-dimensional surface in $\RR^4$. 
As for planar domains considered in Section~\ref{sec:C1-isogeometric-spaces}, the coefficients for the $G^1$
condition are still determined by  the parameterization $\f F^{(S)}$,
in fact by  two of its three components.  In our simplified case, we
are in the following situation. 
\begin{proposition}\label{prop:coeff-surf-case}
Consider surface patches  
\[
	\f F^{(L)} = (\f P^{(L)} , F^{(L)}_3)^T, \; \f F^{(R)} = (\f P^{(R)}, F^{(R)}_3)^T
\]
and functions $ g^{(L)} $, $ g^{(R)} $ fulfilling Definition
\ref{def:G1}. Then  the coefficients $\alpha^{(L)}$, 
$\alpha^{(R)}$ and $\beta$ in  
\eqref{eq:G1-for-F-surf-alpha-beta-gamma} are given by $\alpha^{(S)} (v)  = \gamma(v) 	\bar \alpha^{(S)} (v)
$,  for $S \in \{L,R\}$,  and  $\beta (v)  = \gamma(v) 	\bar \beta
(v) $, where
\begin{equation}
     	\bar   	\alpha^{(S)} (v)=  \det \left[ \begin{array}{ll}
	 D_u \f P^{(S)} (0,v) & D_v \f P^{(S)}
        (0,v)
	\end{array}	 \right],
  \end{equation}
 \begin{equation}
   	\bar  	\beta (v) =  	\det \left[ \begin{array}{ll}
	 D_u \f P^{(L)} (0,v)& 	 D_u \f P^{(R)} (0,v) 
	\end{array}
	 \right],
  \end{equation}
and  $\gamma:[0,1]\rightarrow \RR $ is any  scalar function. In addition, $\gamma(v) \neq 0$ if
and only if $\alpha^{(L)}  (v) \alpha^{(R)} (v) > 0$. Moreover, there exist functions $\beta^{(S)}(v)$, for 
$S\in \{L,R\}$, such that 
\begin{equation}
  \beta (v)= \alpha^{(L)} (v) \beta^{(R)}(v)- \alpha^{(R)} (v)\beta^{(L)}(v). 
\end{equation}
\end{proposition}
Again, we omit the details of the proof, as it is a direct consequence of Proposition \ref{prop:alpha-beta-gamma-are-splines}. 

In analogy to the planar case,    to guarantee optimal approximation properties of the
isogeometric space  we ask some structure for  the trace and derivative
trace  of isogeometric functions, by restricting the  geometry
parametrization to  the class of  analysis-suitable $G^1$-continuous functions. 
The definition is omitted being exactly  the same as
Definition \ref{def:AS-G1}. In the case of 
\eqref{eq:projectable-surface-parametrization}, the parametrization
$\f P^{(S)} $ of the planar projection $\bar \Omega$ of the geometry $\Omega$ needs to be AS
$G^1$. Furthermore, the third component $F^{(S)}_3$ of $\f F^{(S)}$ as
well as the function $\hat\phi^{(S)}$ have to fulfill the same
conditions, for $S\in\{L,R\}$. This means that Remarks
\ref{rem:C1-isoparametric-on-G1-surfaces}--\ref{rem:C1-isoparametric-on-ASG1-surfaces}
extend to surfaces:   $C^1$ isogeometric spaces over  $G^1$  and AS $G^1$  multi-patch
 parametrizations fit in the  isoparametric framework. 
  Eventually, $\phi\in \V^1(\Omega)$ if and only if $\bar
\phi\in \V^1(\bar \Omega)$, where $\bar
\phi \circ \f P^{(S)} =  \phi \circ \f F^{(S)}$. Then, results   of Section~\ref{two-patch} apply directly to AS $G^1$  surfaces, at least  when
the  geometry parametrization fulfills
\eqref{eq:projectable-surface-parametrization}. 

\begin{remark}
 As a consequence of Proposition~\ref{prop:coeff-surf-case},
 $C^1$ isogeometric spaces over $G^1$ multi-patch surface
 parametrizations fit in the  isoparametric framework. 
\end{remark}

\section{NURBS spaces}
\label{sec:beyond-splines}

The results presented in Sections
\ref{sec:C1-isogeometric-spaces}--\ref{two-patch} are not limited to
spline spaces, but, for example, are  generalizable  to NURBS. 
One possibility is to generate rational planar  geometries from AS
$G^1$ surface patches.  A rational patch 
\begin{equation}
    	\f F^{(i)}: \widehat\Omega\rightarrow \Omega^{(i)} \subset \RR^2
\end{equation}
with 
\begin{equation}\label{eq:rational}
    	\f F^{(i)} = \left( \frac{F^{(i)}_1}{F^{(i)}_3}, \frac{F^{(i)}_2}{F^{(i)}_3} \right)^T
\end{equation}
can be interpreted as a surface patch in $\RR^3$, with 
\begin{equation}\label{eq:surf}
    	\widetilde{\f F}^{(i)} = \left( F^{(i)}_1, F^{(i)}_2, F^{(i)}_3 \right)^T,
\end{equation}
when transformed to homogeneous coordinates. Recall, that two points $\widetilde{\f F}$, $\widetilde{\f F}^\prime$ in homogeneous coordinates correspond to the same point in Euclidean space, if there exists a $\lambda \neq 0$ such that $\widetilde{\f F} = \lambda \widetilde{\f F}^\prime$. 
An AS $G^1$ surface with surface patches \eqref{eq:surf} generates a
rational AS $G^1$ geometry via \eqref{eq:rational}. Therefore this
construction  fits into the framework of Section~\ref{surfaces}. For the sake of
brevity, we only present one example (Figure
\ref{fig:QuartOfCircleResults}) in the next Section
\ref{sec:multi-patch-ASG1} and postpone  further studies. 
We want to point out that not all multi-patch NURBS geometries fit into this framework. 
The surface representation in \eqref{eq:surf} may not be $G^1$ (or not even $C^0$) but the rational 
patches representing the graph still join $G^1$. Such a configuration is given for the circle presented in Figure~\ref{fig:CircleResults}. 
However, in that example the graph surface is not analysis-suitable $G^1$. 

\section{Numerical tests}
\label{numerical_tests}

In this section, we present numerical tests which illustrate the previous theoretical results. The simulations have been obtained by
the isogeometric analysis library IGATOOLS, see \cite{pauletti-martinelli-cavallini-14} for an introduction.  For related numerical studies see \cite{nguyen2014comparative,kapl-vitrih-juttler-birner-15,nguyen2016,kapl-buchegger-bercovier-juttler-16}. 

\subsection{Model problem}
We  consider the following bilaplacian problem on $\Omega \subset
\RR^2$, where the unknown is denoted by $\disp$ and the data by
$\force$, 
\begin{equation}
	\label{eq:biharmonicPb}
	\left\{
	\begin{array}{rcll}
		\Delta^2 \, \disp & = & \force &  \Body, \\
		\disp & = & 0 &  \partial\Body, \\
		{\nabla} \disp \cdot \normal & = & 0 & \partial\Body. \\
 	\end{array}
	\right.
\end{equation}
In our tests the data $\force$ is selected in order to have an
analytic exact solution $\disp$.  Here $\normal$ is the unit normal vector to the boundary $\partial \Omega$. 
Let 
\[
	\mathcal{U}_0 = \{\dispTF \in  H^2(\Body), \dispTF= 0 \text{ on }\partial \Omega \text{ and } {\nabla} \dispTF \cdot \normal = 0 \text{ on }\partial \Omega\},
\]
the variational form of \eqref{eq:biharmonicPb} is, find $\disp \in \mathcal{U}_0$, such that 
\begin{eqnarray} \label{eq:biharmonicPbVarForm}
	\int_{\Body} \Delta \disp \Delta \dispTF \;  \mathrm{d}x \, \mathrm{d}y = \int_{\Body} \force \dispTF \;  \mathrm{d}x \, \mathrm{d}y , \; \forall \dispTF \in \mathcal{U}_0.
\end{eqnarray}

In what follows, we consider two types of geometries, the ones which are analysis-suitable $G^1$ and the ones which are not. Let us start with the analysis-suitable $G^1$ geometries.

\subsection{Analysis-suitable $G^1$ geometries}

We consider five different analysis-suitable $G^1$ geometry
mappings with different numbers of patches. For all cases, we select
the degrees $p = 3, 4, 5$ and  regularities $1 \leq r \leq p-1$.

\subsubsection{Two-patch geometry (L-shape)}

\begin{figure}[h!]
	\centering
		\includegraphics[width=\textwidth]{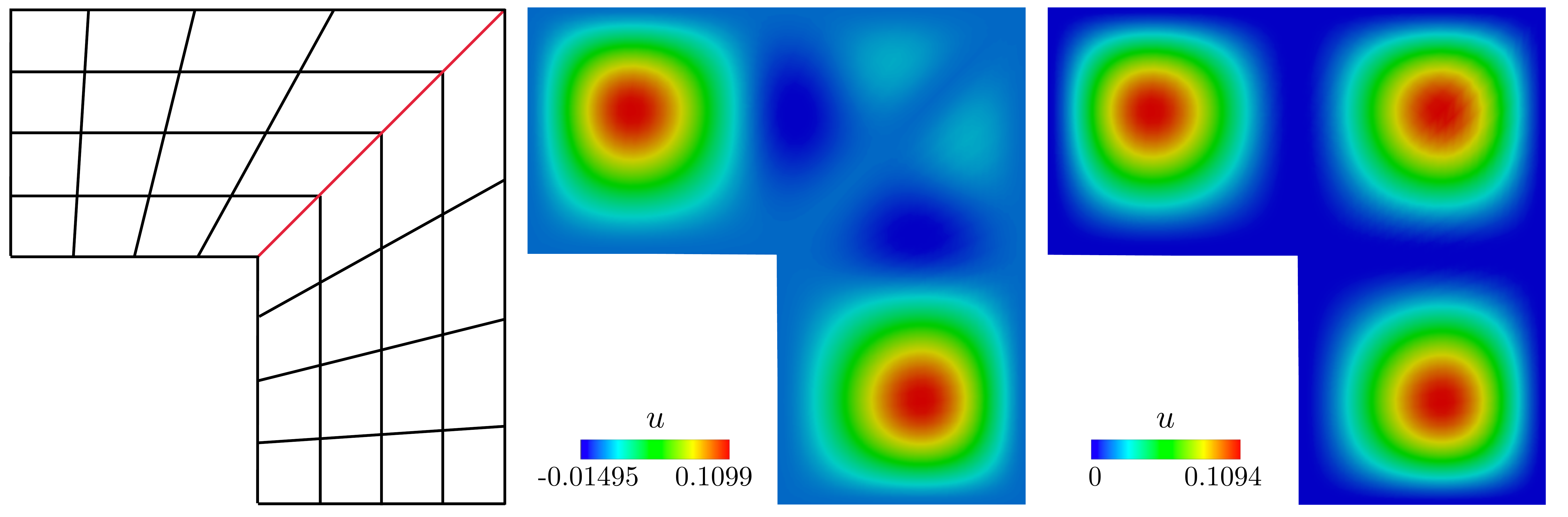}
	\caption{Two-patch L-shaped domain (left), a solution
          affected by $C^1$ locking  for $p = 3$, $r = 2$ (center),
          and a correct numerical solution for $p = 3$, $r = 1$ (right).}
	\label{fig:LResults}
\end{figure}

We start with the L-shaped geometry consisting of two patches. The L-shaped domain
is given in Figure~\ref{fig:LResults} (left), where the red edge is the
common edge of both patches. Here the parametrization of both patches is bilinear. This is
obviously  an AS $G^1$ geometry (recall Proposition \ref{prop:bilinear}). Using
Theorem~\ref{thm:trace-for-AS-G1}, the optimal convergence is achieved
if and only if $r \leq p-2$,  in agreement with the results of
Section \ref{AS-G1-parametrization}. If we focus on
degree $p=3$,  $C^1$ locking is evident when the regularity equals to $r=p-1 = 2$, see Figure~\ref{fig:LResults} (center). In particular, the solution on the interface line equals to $0$. In order to circumvent $C^1$ locking, we decrease the regularity, see Figure~\ref{fig:LResults} (right). Figure~\ref{fig:LResultsConv} gives the convergence curves for degrees $p = 3, 4$ and~$5$. We obtain expected results for all degrees, \emph{i.e.} we have convergence (which is optimal) if and only if $r+2 \leq p$. 

\begin{figure}[h!]
	\centering
	\begin{tabular}{c}
		\includegraphics[width=0.5\textwidth]{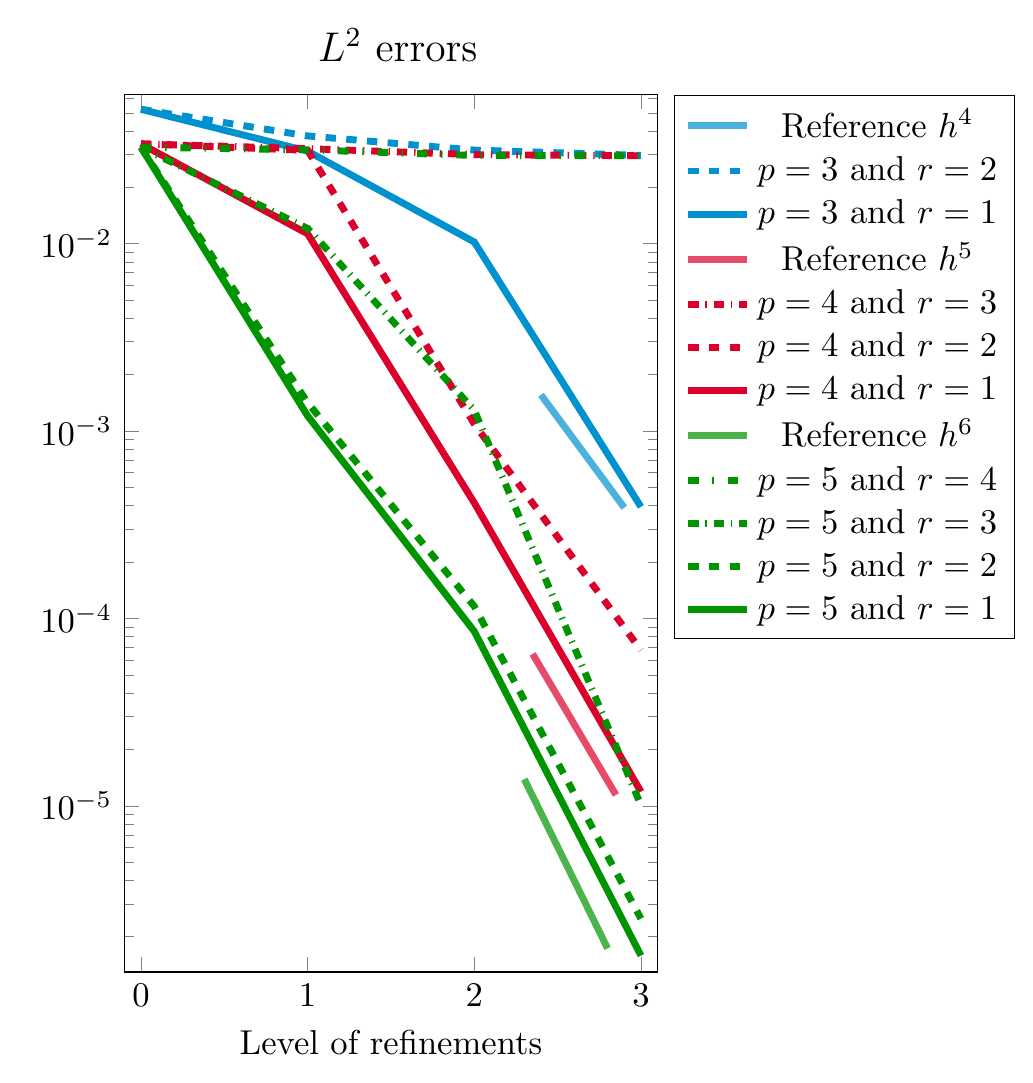} 
	\end{tabular}
	\caption{Convergence results for the L-shaped domain.}
	\label{fig:LResultsConv}
\end{figure}

\subsubsection{Multi-patch geometries (triangle, quarter of a circle and rectangle)}
\label{sec:multi-patch-ASG1}

\begin{figure}[h!]
	\centering
		\includegraphics[width=\textwidth]{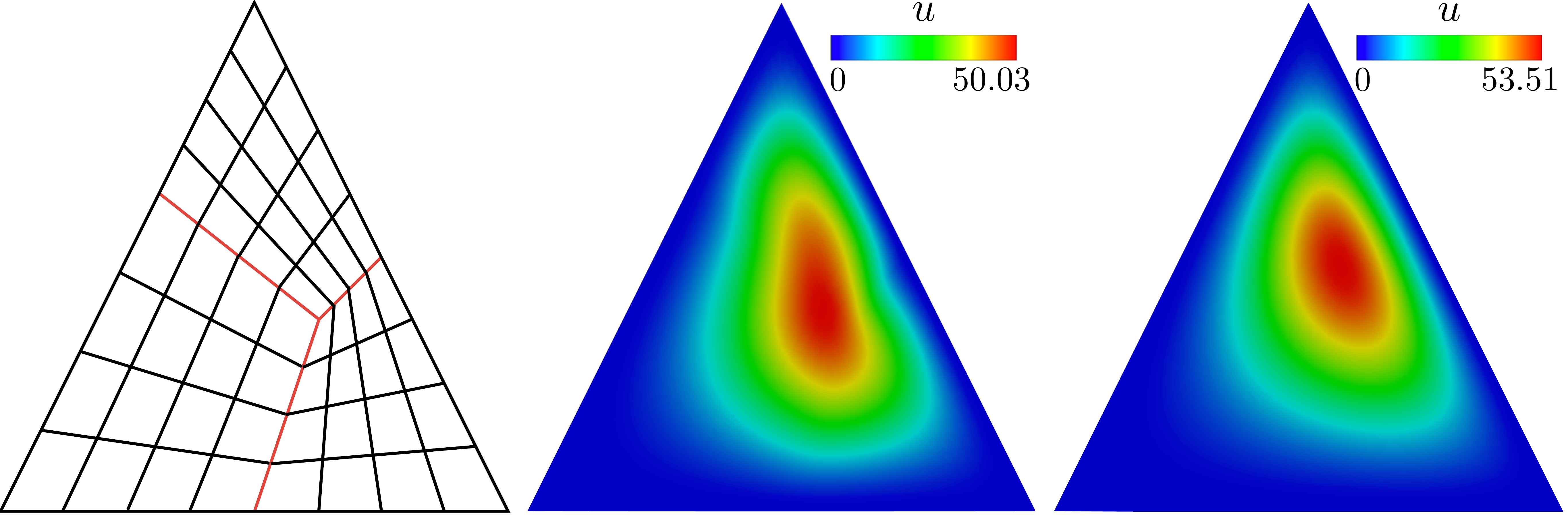}
	\caption{Three-patch triangle (left), a solution
          affected by $C^1$ locking for $p = 3$, $r = 2$ (center), and a correct numerical solution for $p = 3$, $r = 1$ (right).}
	\label{fig:TriangleResults}
\end{figure}

\begin{figure}[h!]
	\centering
		\includegraphics[width=\textwidth]{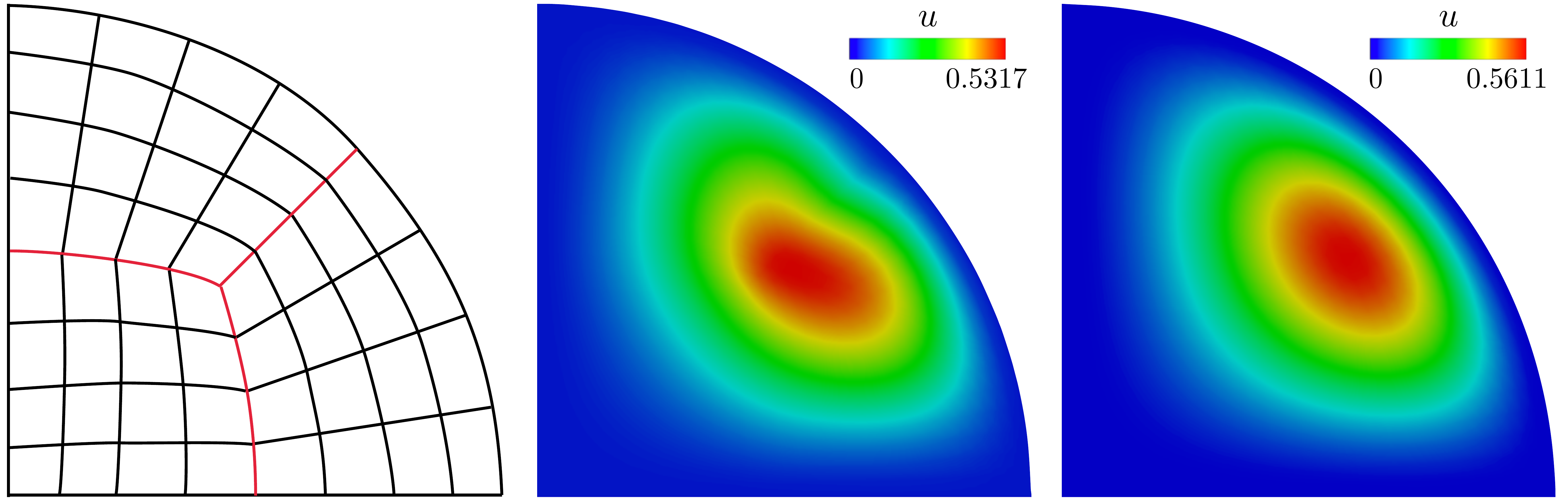}
	\caption{Three-patch quarter of a circle (left), a solution
          affected by $C^1$ locking for $p = 3$, $r = 2$ (center), and a correct numerical solution for $p = 3$, $r = 1$ (right).}
	\label{fig:QuartOfCircleResults}
\end{figure}

In what follows, we consider cases with three or more patches, starting with an example with
three bilinear patches forming a triangle (see
Figure~\ref{fig:TriangleResults} (left))  and another with three bi-quadratic
patches forming the quarter of a circle (see
Figure~\ref{fig:QuartOfCircleResults} (left)). Both are  AS $G^1$ geometries. 
Note that the quarter of the circle is composed of NURBS  patches, and
is obtained following Section~\ref{sec:beyond-splines} construction,
from a  geometry parametrization that is  an AS $G^1$ surface in
homogeneous coordinates. The details of this construction are
presented in the appendix.  Figures~\ref{fig:TriangleResults} (center) and
\ref{fig:QuartOfCircleResults} (center) show $C^1$ locking which appears with $p = 3$ and $r = 2$ that we can circumvent by reducing the regularity, see  Figures~\ref{fig:TriangleResults} (right) and \ref{fig:QuartOfCircleResults} (right). The expected convergence orders are given in Figure~\ref{fig:3PatchesConv}.

\begin{figure}[h!]
	\centering
	\begin{tabular}{cc}
		\includegraphics[width=0.5\textwidth]{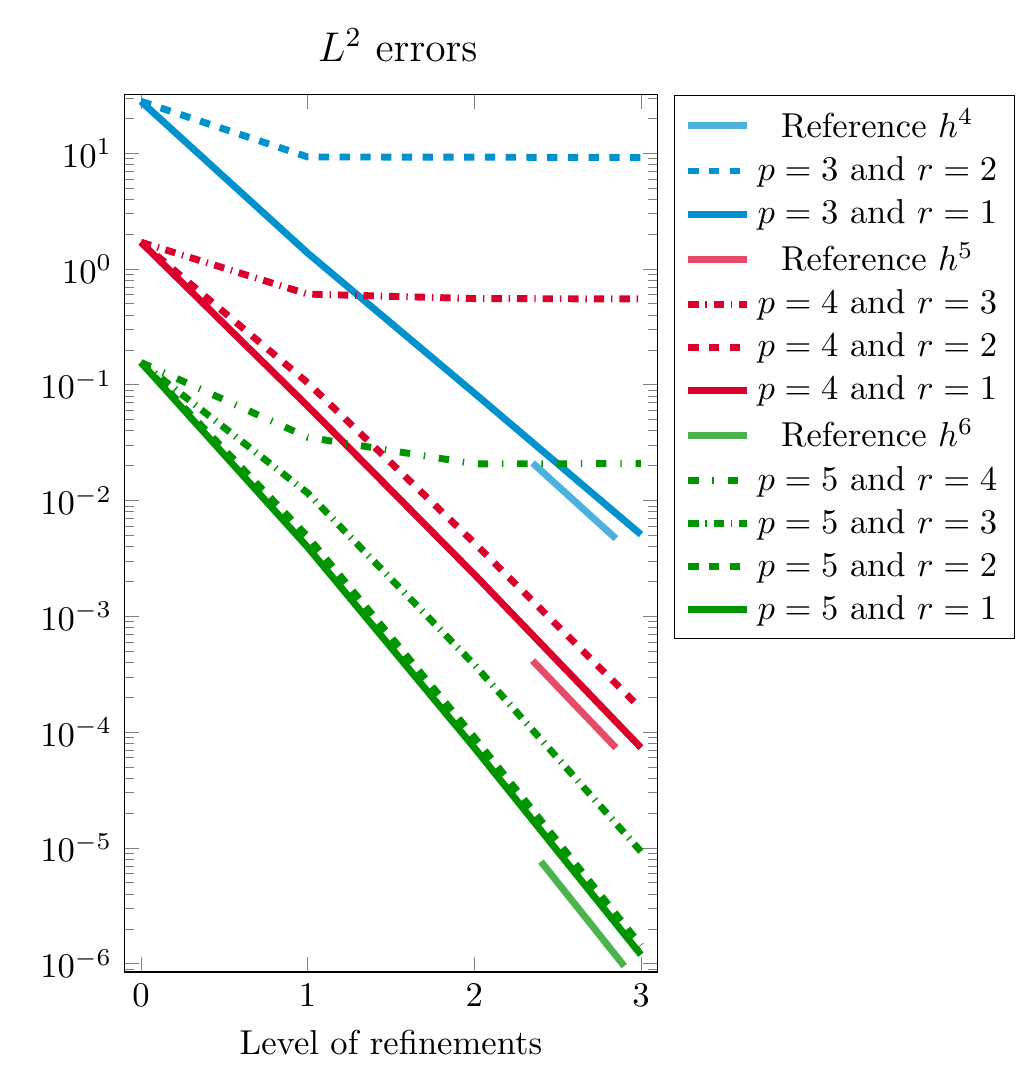} & 
		\includegraphics[width=0.5\textwidth]{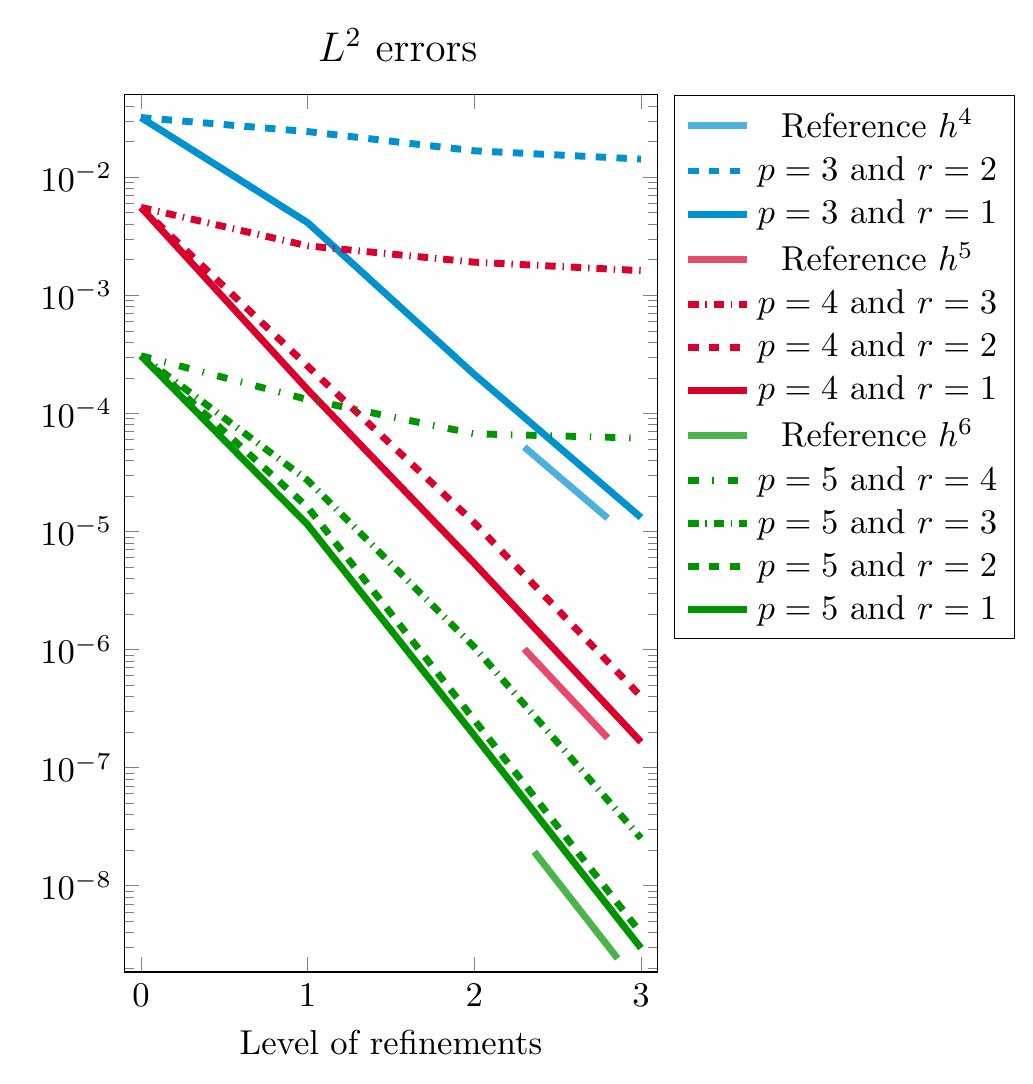} \\
	\end{tabular}
	\caption{Convergence results for the triangle (left) and the quarter of circle (right).}
	\label{fig:3PatchesConv}
\end{figure}

\begin{figure}[h!]
	\centering
		\includegraphics[width=\textwidth]{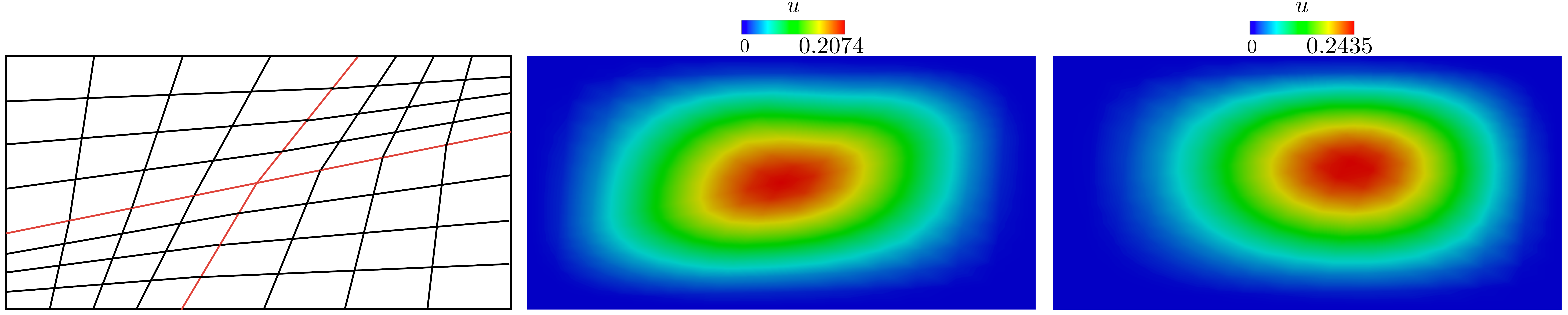}
	\caption{Four-patch rectangle (left), a solution
          affected by $C^1$ locking for $p = 3$, $r = 2$ (center), and a correct numerical solution for $p = 3$, $r = 1$ (right).}
	\label{fig:RectangleResults}
\end{figure}

Next we consider a rectangle composed of four patches, see
Figure~\ref{fig:RectangleResults} (left). As shown in this figure, we
consider a particular case where two interfaces are collinear. 
This configuration is among the ones analyzed in \cite[Section
11.2.3]{bercovier2014smooth}. The results are presented in
Figure~\ref{fig:RectangleResults} for degree $p=3$ and regularity $r =
2$ (center) and $r = 1$ (right). Figure~\ref{fig:3and5PatchesConv}
(left) gives the convergence orders, which are as expected.  
\begin{figure}[h!]
	\centering
		\includegraphics[width=\textwidth]{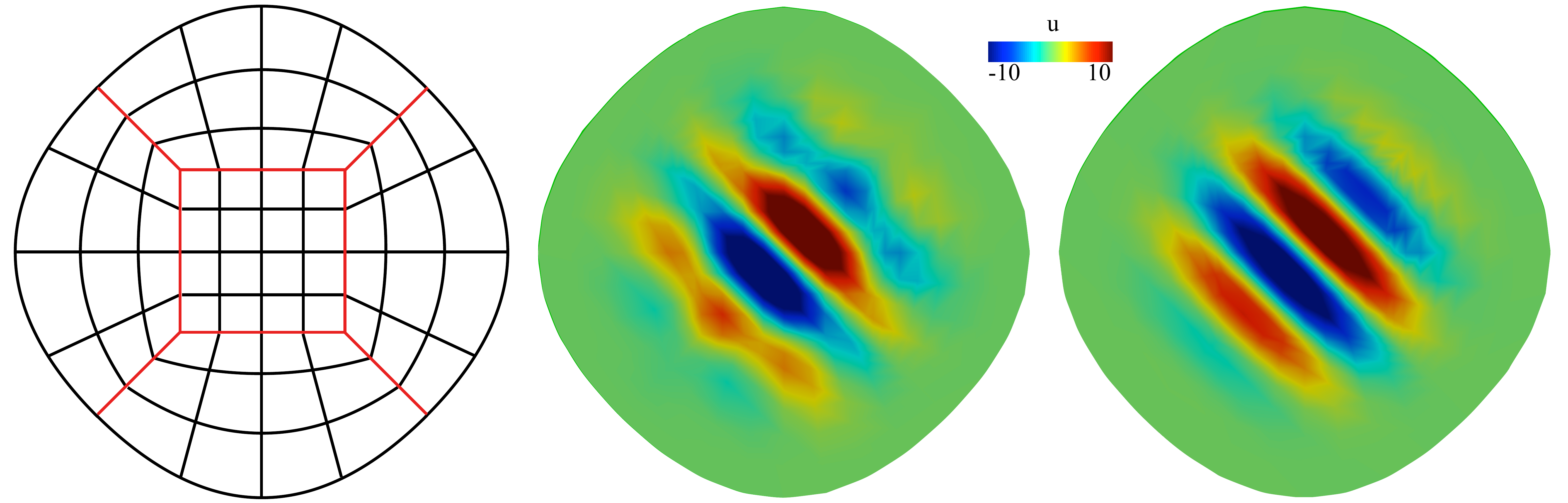}
	\caption{Five-patch simply-connected
domain with smooth boundary (left), a solution
          affected by $C^1$ locking for $p = 3$, $r = 2$ (center), and a correct numerical solution for $p = 3$, $r = 1$ (right).}
	\label{fig:5PatchSmoothResults}
\end{figure}

The  final and most relevant example of AS $G^1$
geometry is reported in Figure~\ref{fig:5PatchSmoothResults}
(left). This is a  five patch decomposition of a simply-connected
domain with smooth boundary. The parametrization of each patch is
bi-quadratic, and the domain boundary  is $C^1$. Given the boundary
control points and parametrization, the interior contol points have been selected in
order to fulfil the AS $G^1$ conditions. Unlike \cite{peters-PhD},
here not only the mesh curves, i.e., patch
interfaces, but also their parametrization need to be chosen properly. We refer to the
Appendix for the complete description.  
The results are presented in
Figure~\ref{fig:5PatchSmoothResults} for degree $p=3$ and regularity $r = 2$ (center) and $r = 1$ (right).  Figure~\ref{fig:3and5PatchesConv} (right) gives the convergence orders, which are again as expected. 

\begin{figure}[h!]
	\centering
	\begin{tabular}{cc}
		\includegraphics[width=0.5\textwidth]{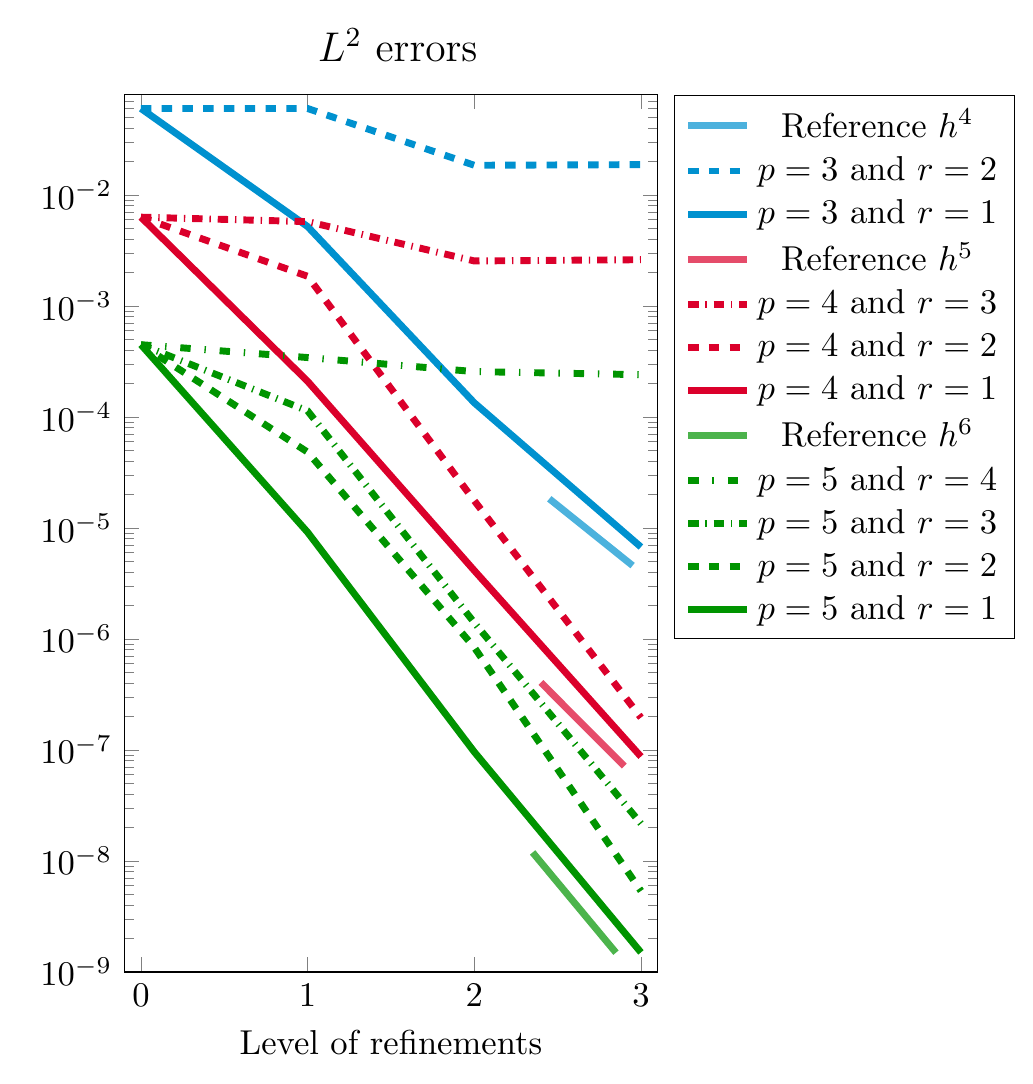} & 
		\includegraphics[width=0.5\textwidth]{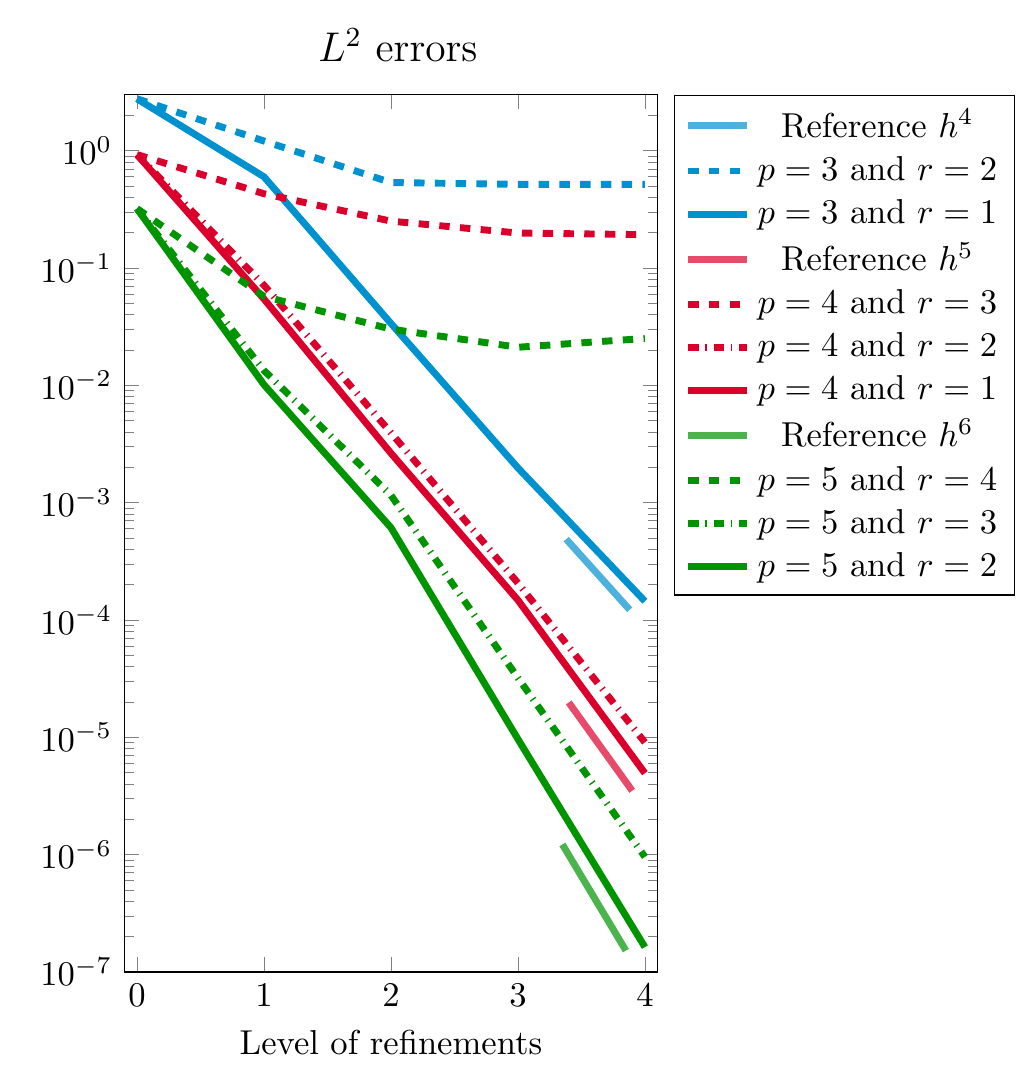}\\
	\end{tabular}
	\caption{Convergence results for the four-patch rectangle (left) and the five-patch simply-connected
domain with smooth boundary (right).}
	\label{fig:3and5PatchesConv}
\end{figure}

\subsection{Non-analysis-suitable $G^1$ geometries}

In this section we study two different examples of non-analysis-suitable $G^1$ geometries. First we 
consider a two-patch parametrization of a rectangle, where we compare a distorted parametrization with 
the undistorted identity mapping, see Figure~\ref{fig:IdDist}. 
As a second example we consider a five-patch circle.

\subsubsection{Two-patch geometry (quadratically distorted rectangle)}

\begin{figure}[h]
	\centering
		\includegraphics[width=\textwidth]{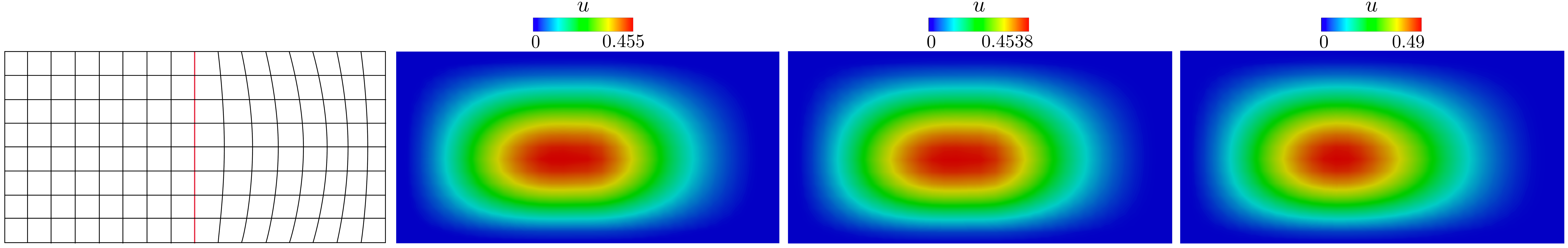}
	\caption{Non-AS $G^1$ rectangle (left),  two solutions
          affected by $C^1$ locking for $p=3$, $r=2$ (center-left) 
          and $p = 3$, $r = 1$ (center-right), as well as a correct numerical solution for $p = 4$, $r = 1$ (right).}
	\label{fig:RectDistResults}
\end{figure}

We consider a rectangle domain $[-1,1]\times [0,1]$ formed by two patches, where the left  one is
linear (identity) and right one is quadratic in the horizontal
direction and linear in the vertical, see
Figure~\ref{fig:RectDistResults} (left). This case illustrates the configuration covered by 
Theorem~\ref{thm:trace-for-general-parametrization} and Corollary~\ref{cor:convergence-general-parametrization}.
As shown in Figure~\ref{fig:RectDistResults}, with degree $p = 3$, $C^1$ locking is manifested for both  $r = 2$ and $r = 1$, see
respectively the second and the third columns of the
figure.  In order to have convergence, a degree of at least $p=4$ has to
be selected, see Figure~\ref{fig:RectDistResults} (right).

However, as anticipated in
Corollary~\ref{cor:convergence-general-parametrization}, we cannot
expect optimal convergence.  This is a direct implication of Theorem
\ref{thm:trace-for-general-parametrization}, which for this case ($p_{\alpha} =
2$ and $\beta^{(R)}=0$) states that for all $\phi \in \V^1$
\begin{equation}\label{eq:deriv-degree2}
   \left . \frac{\partial \phi }{\partial
    x}\right |_{\Gamma}  \in   \S^{p-p_{\alpha}}_{r} (\Gamma)= \S^{2}_{1} (\Gamma).
\end{equation}
Assume that the exact solution
$\disp$ is  smooth enough
and let $\disp_h$ be the numerical solution. Then, 
using \eqref{eq:deriv-degree2}, together with the  usual approximation estimates in Sobolev norms and 
the trace inequality for Sobolev spaces, gives 
 \begin{displaymath}
    \begin{aligned}
  C_{approx}  h^{2.5} & \simeq  \left \| \frac{\partial \disp }{\partial  x} -\frac{\partial \disp_h }{\partial
    x}\right \|_{H^{\frac{1}{2}}(\Gamma)} \leq C_{trace} \| \disp-\disp_h\|_{H^2(\Omega)},
    \end{aligned}
\end{displaymath}
instead of the optimal order of convergence, that is $h^3$ when measuring 
the error in the $H^2(\Omega)$-norm. 

\begin{figure}[h]
	\centering
		\includegraphics[width=0.6\textwidth]{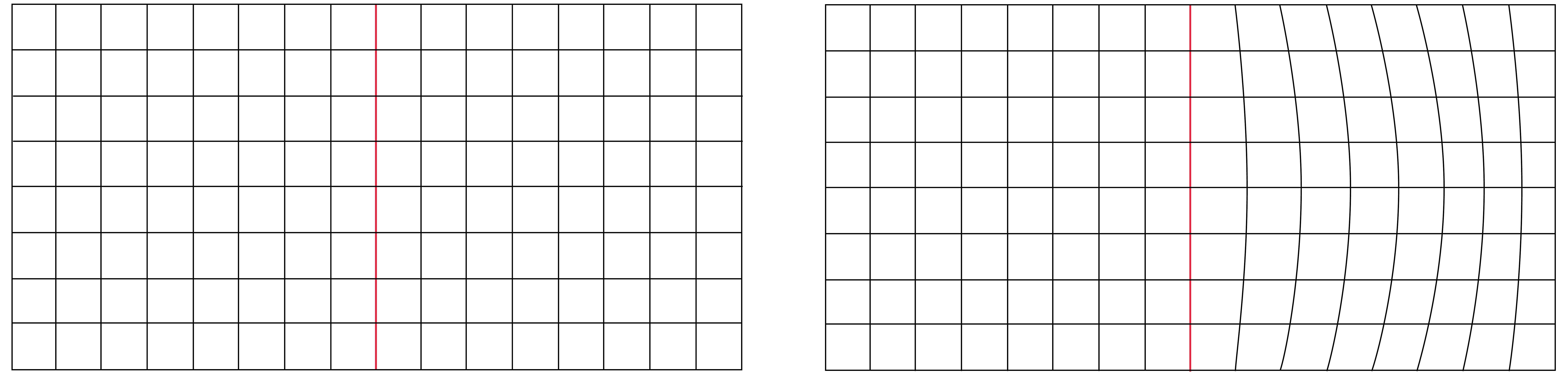}
	\caption{Two-patch rectangle with identity mapping (left), two-patch rectangle with a quadratic distorted patch (right).}
	\label{fig:IdDist}
\end{figure}
In Figure~\ref{fig:IdDistErrorL2H2} we report the convergence results
and, due to the specificity of this case, both $L^2$ and $H^2$ norms
are plotted. 
While $C^1$ locking is easily recognized, from these numerical tests
it is difficult to measure the expected sub-optimality of the
asymptotic behavior.  This is likely a numerical artifact  due to the imposition of the $C^1$ constraint in our
implementation, which is discussed in Section
\ref{sec:numerical-implement}.  We further analyze this example in
Figure~\ref{fig:IdDistUndefComparison}, where  we compute, for $p = 4$ and
$r = 1$,  the error only on the left patch, that is  $[-1,0]\times [0,1]$, and compare the
results for this geometry and for a  reference geometry formed by  the identity mapping on both  
patches (see Figure~\ref{fig:IdDist}). Note that on the left patch
both parametrizations are the identity mapping.

\begin{figure}[h]
	\centering
	\begin{tabular}{cc}
		\includegraphics[width=0.51\textwidth]{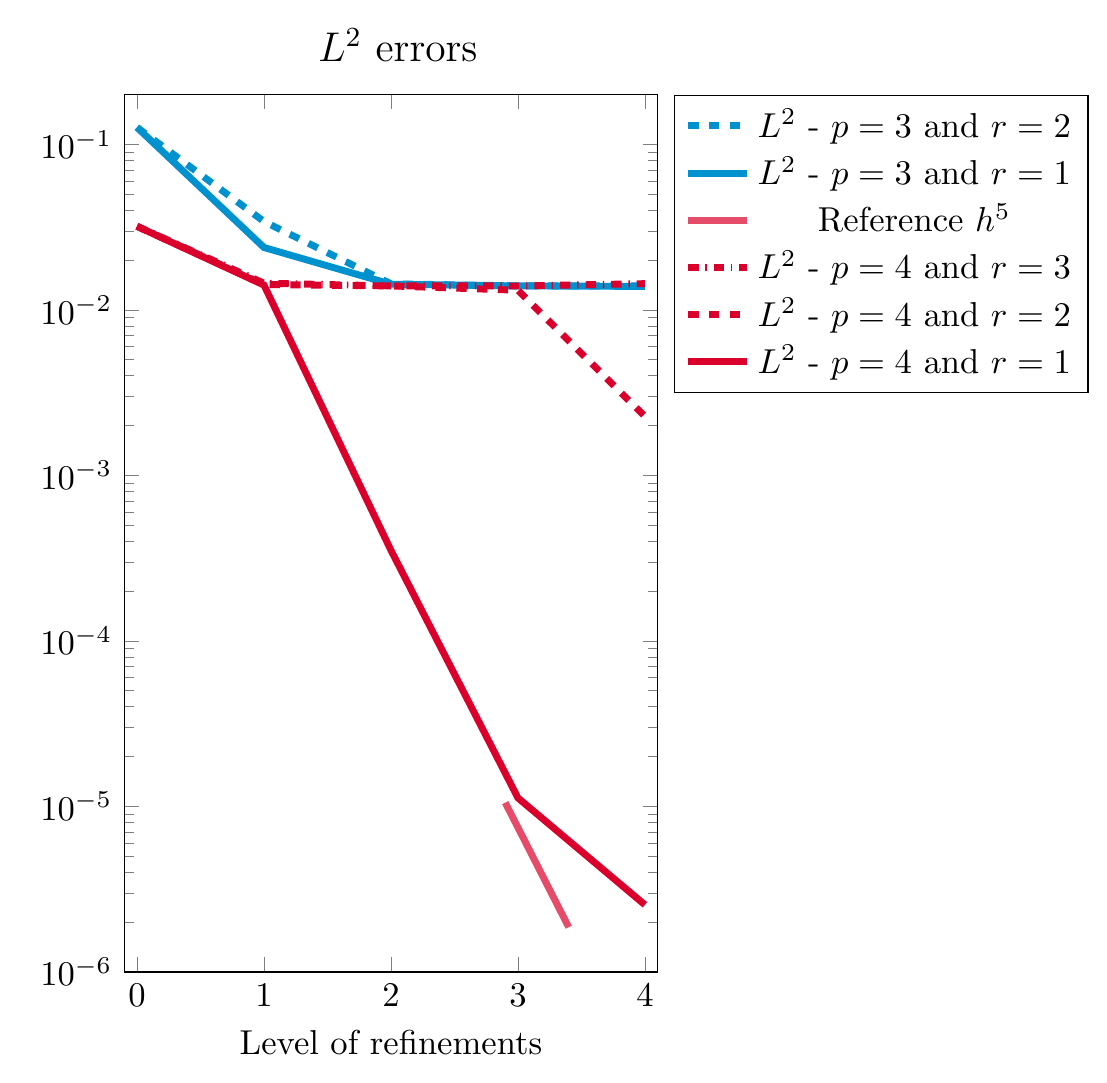}
           &
		\includegraphics[width=0.51\textwidth]{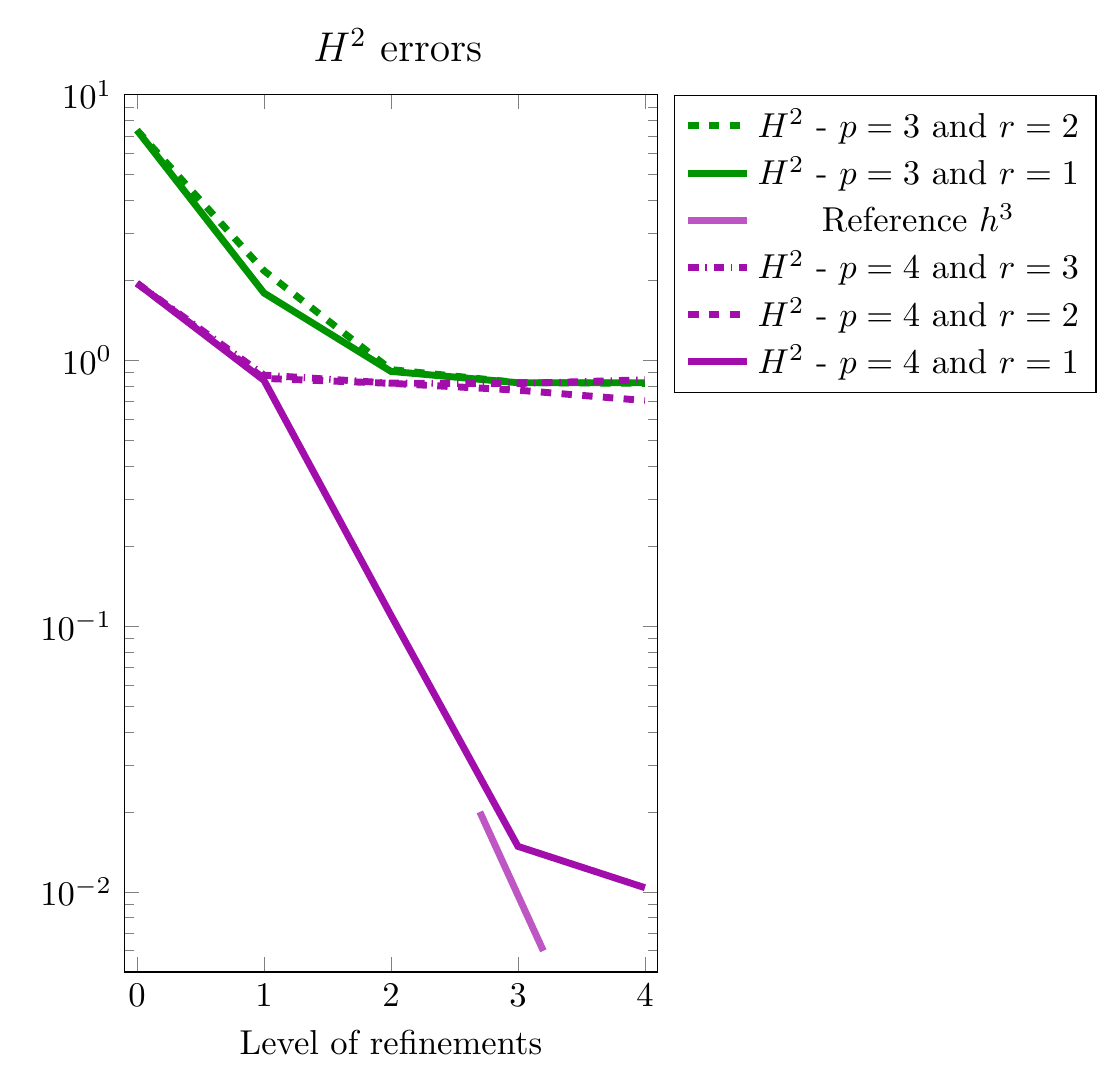} \\
	\end{tabular}
	\caption{Convergence results for the  non-AS $G^1$ two-patch
          rectangle, with error in $L^2$  (left)  and $H^2$
          (right).}
	\label{fig:IdDistErrorL2H2}
\end{figure}
\begin{figure}[h]
	\centering
	\begin{tabular}{cc}
		\includegraphics[width=0.49\textwidth]{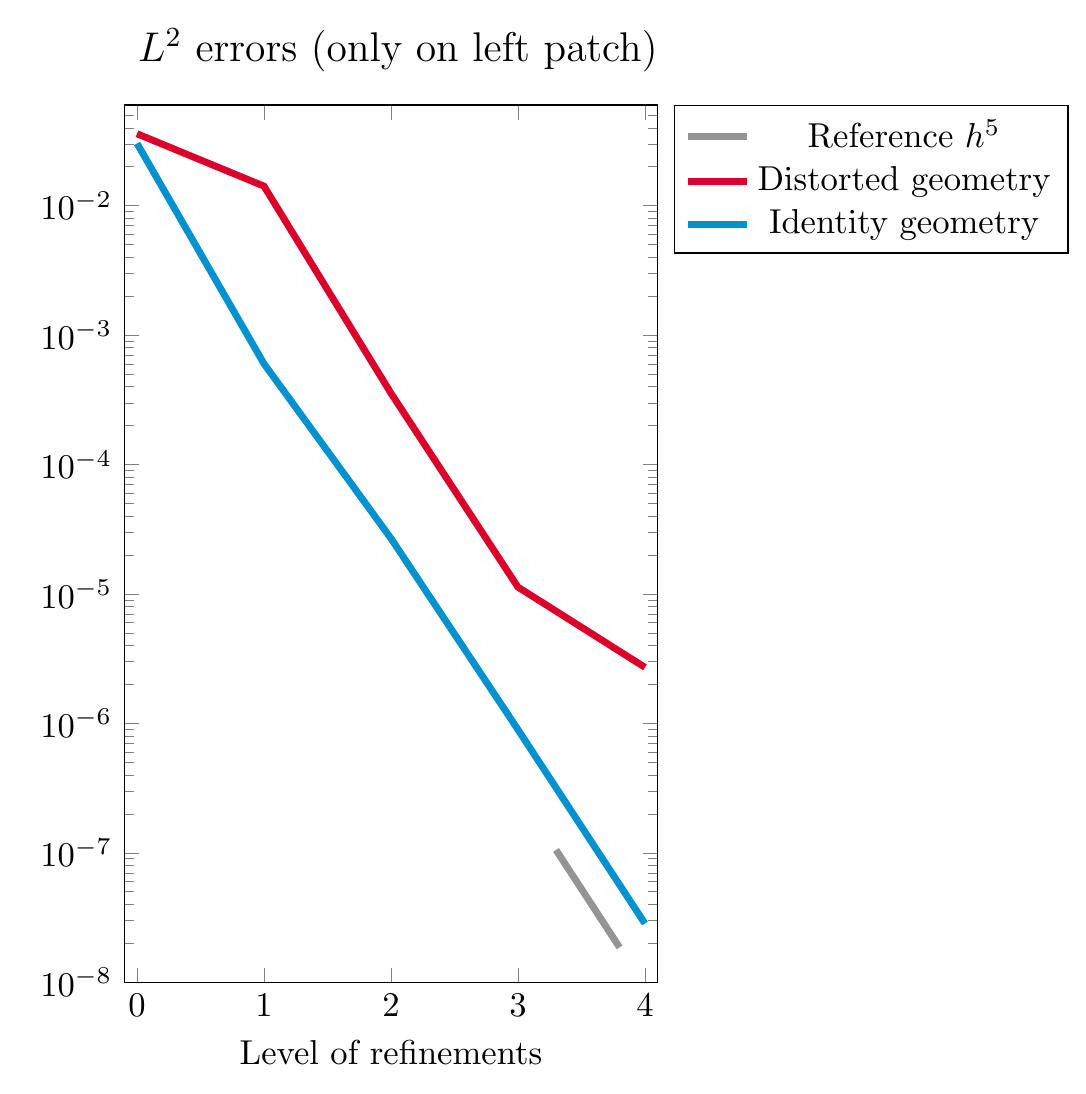} & 
		\includegraphics[width=0.49\textwidth]{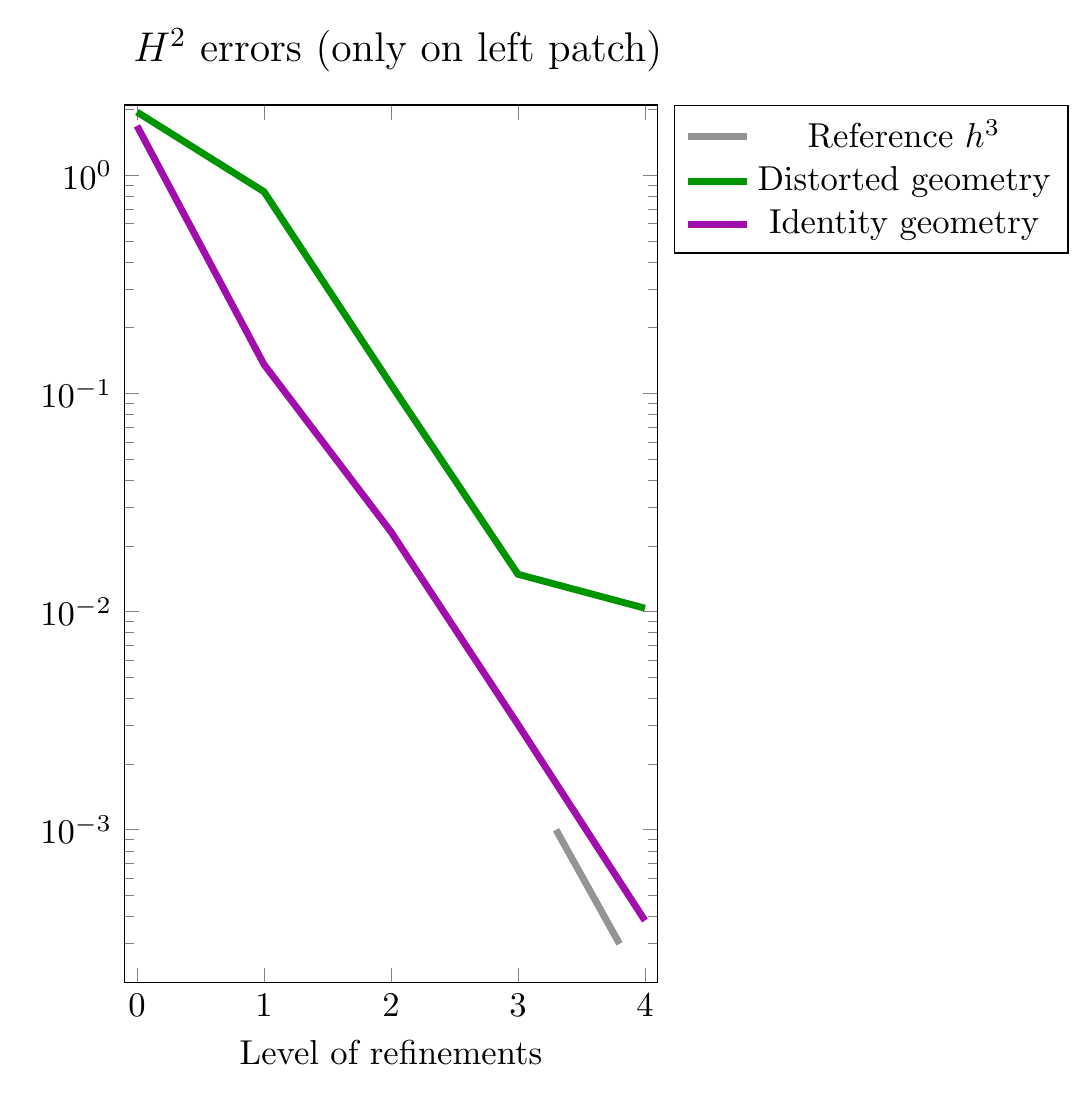} \\
	\end{tabular}
	\caption{Comparison of convergence between the two geometries
          given in Figure~\ref{fig:IdDist} for $p = 4$ and $r = 1$;
          the error is computed only in the left subdomain $ [-1,0] \times
          [0,1]$ in $L^2$  (left plot)  and $H^2$ (right plot).}
	\label{fig:IdDistUndefComparison}
\end{figure}

\subsubsection{Multi-patch geometry (circle)}

\begin{figure}[h]
	\centering
		\includegraphics[width=\textwidth]{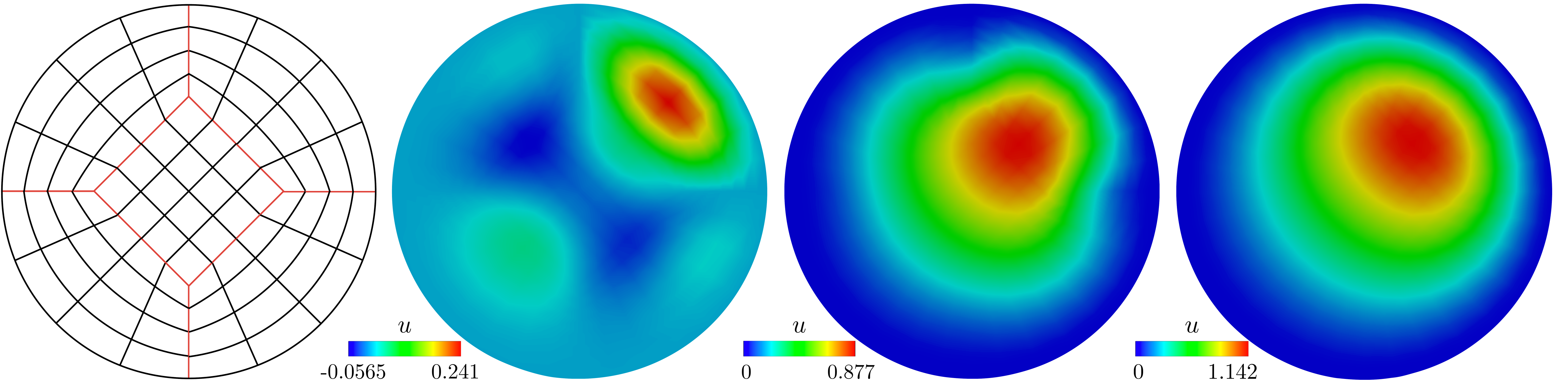}
	\caption{Five-patch circle (left),  two solutions
          affected by $C^1$ locking for $p = 3$,
          $r = 2$ (center-left) and for $p = 3$, $r = 1$ (center-right), as well as a correct numerical solution for $p = 4$, $r = 1$ (right).}
	\label{fig:CircleResults}
\end{figure}

In the last example, we study an exact circle composed of five patches, given in
Figure~\ref{fig:CircleResults} (left). Here, we are interested
in testing  rational parametrizations that are beyond the framework
presented in Section \ref{sec:beyond-splines},  since their homogeneous representation is not an 
analysis-suitable $G^1$ surface. Even though  our theory does not
apply, the numerical results obtained are consistent with 
 our findings. For degree $p = 3$ the numerical solution suffers
 of $C^1$ locking, as one can see in Figure~\ref{fig:CircleResults} for regularity
 $r = 2$ (center-left)
 and $r=1$ (center-right). For degree $p=4$ and regularity $r=1$, see Figure~\ref{fig:CircleResults} (right), we 
observe convergence to a solution.
 Figure~\ref{fig:RectangleDistCircleConv} (right) gives the convergence
 behavior for degrees $p = 3, 4$ and $5$. Sub-optimality in all situations is manifested.

\begin{figure}[h]
	\centering
	\begin{tabular}{cc}
		\includegraphics[width=0.5\textwidth]{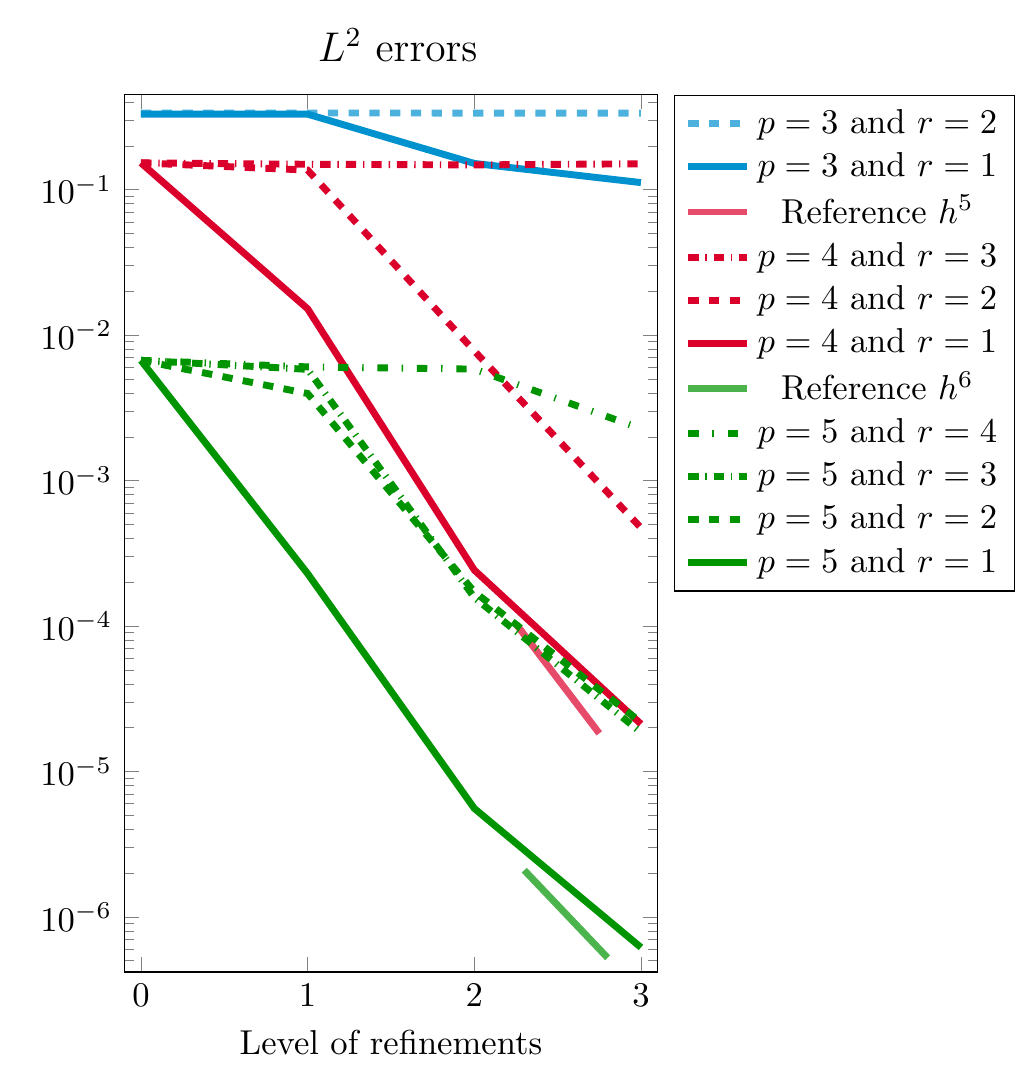} \\
	\end{tabular}
	\caption{Convergence results for the five-patch circle.}
	\label{fig:RectangleDistCircleConv}
\end{figure}

\subsection{Numerical implementation of  $C^1$ continuity} \label{sec:numerical-implement}

In this section, we describe the numerical implementation of $C^1$
continuity that we have used in order to obtain all the numerical
examples given previously. 
Let $\matrixSystem$ and $\rhs$, be the
matrix and the right and side of the system corresponding to the
variational formulation \eqref{eq:biharmonicPbVarForm}, where no
boundary condition or continuity condition are included. Furthermore, let $\ConstraintDer$ be the $C^1$ constraint matrix in symmetric form, i.e.
\[
	(\ConstraintDer)_{i,j} = \int_\Gamma [ \! [ \vec{\nabla} \phi_i \cdot \vec{n}  ] \! ]  \,   [ \! [ \vec{\nabla} \phi_j \cdot \vec{n} ] \! ] , 
\]
where $\vec{n}  $ is a normal unitary vector and $[ \! [ \cdot  ] \! ]
$ is the jump at the patch interface $\Gamma$. 
Let $\NSBCCont$ be the  change of basis matrix from the fully
unconstrained space to the subspace fulfilling boundary condition
and $C^0$ continuity. $\NSBCCont$ can be seen as obtained from
the identity $\Id_{Nunc}$ --~where $Nunc$ is the number of degrees of
freedom without constraints~-- by removing the columns with index
corresponding to the degrees of freedom of the boundary conditions and
by summing  the columns of degrees of freedom that are shared on $\Gamma$. 
Since it is not trivial to  compute a $C^1$ continuous basis
analytically,   we operate numerically by computing the null-space  
\begin{equation}
  \label{eq:nullspace}
  	\NSBCContDer  = \text{null} (\NSBCCont^T \ConstraintDer
        \NSBCCont). 
\end{equation}
Then we solve the following problem: find $\solBCContDer$ such that
\begin{equation} \label{systemBCContDer}
  \NSBCContDer^T \NSBCCont^T \matrixSystem \NSBCCont \NSBCContDer \, \solBCContDer = \NSBCContDer^T \NSBCCont^T \rhs,
\end{equation}
and obtain the solution in the unconstrained initial spline basis as $\sol = \NSBCCont \NSBCContDer \solBCContDer$.

 The numerical computation of \eqref{eq:nullspace}  is a hard task for
 non AS  $G^1$ geometries.  Indeed, the non-zero eigenvalues of  $\NSBCCont^T \ConstraintDer
        \NSBCCont$ are not well separated from  the eigenvalues that are numerically zero (close to machine precision), as
        shown in Figure~\ref{fig:EVDistCase} for the  distorted
        rectangular domain in   Figure~\ref{fig:RectDistResults}
        (left), with $p = 4$, $r = 1$.  For the next refinement step, the distinction may not be possible anymore. By comparison,
        Figure~\ref{fig:EVLshapeCase} shows the eigenvalues for the
        L-shaped domain with the same degree and regularity.  In this case, it is easier to distinguish between numerical zeros (around $10^{-14}$) and non-zero eigenvalues (above $10^{-2}$). The two
        configurations are structurally equivalent, i.e. the topology
        of the control point grid is the same.  
\begin{figure}[h]
	\centering
		\includegraphics[width=0.6\textwidth]{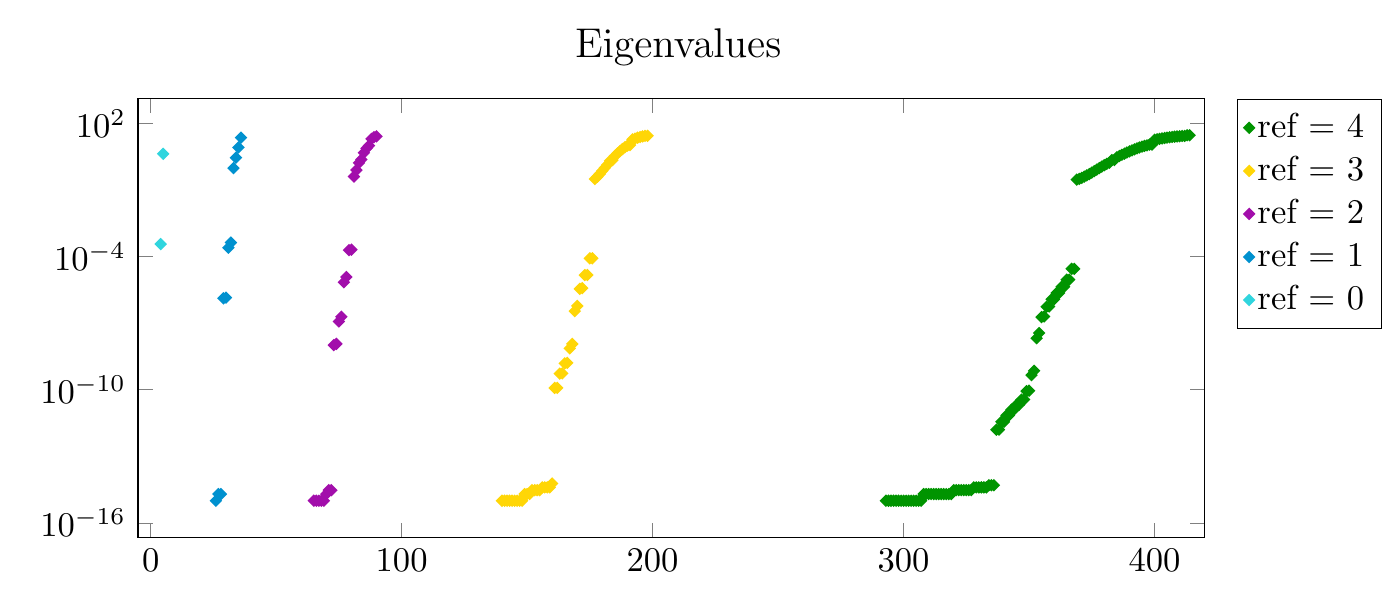}
	\caption{Eigenvalues of the domain given in
          Figure~\ref{fig:RectDistResults} (left) for p = 4 and r = 1,
        and different $h$-refinement levels.}
	\label{fig:EVDistCase}
\end{figure}

\begin{figure}[h]
	\centering
		\includegraphics[width=0.6\textwidth]{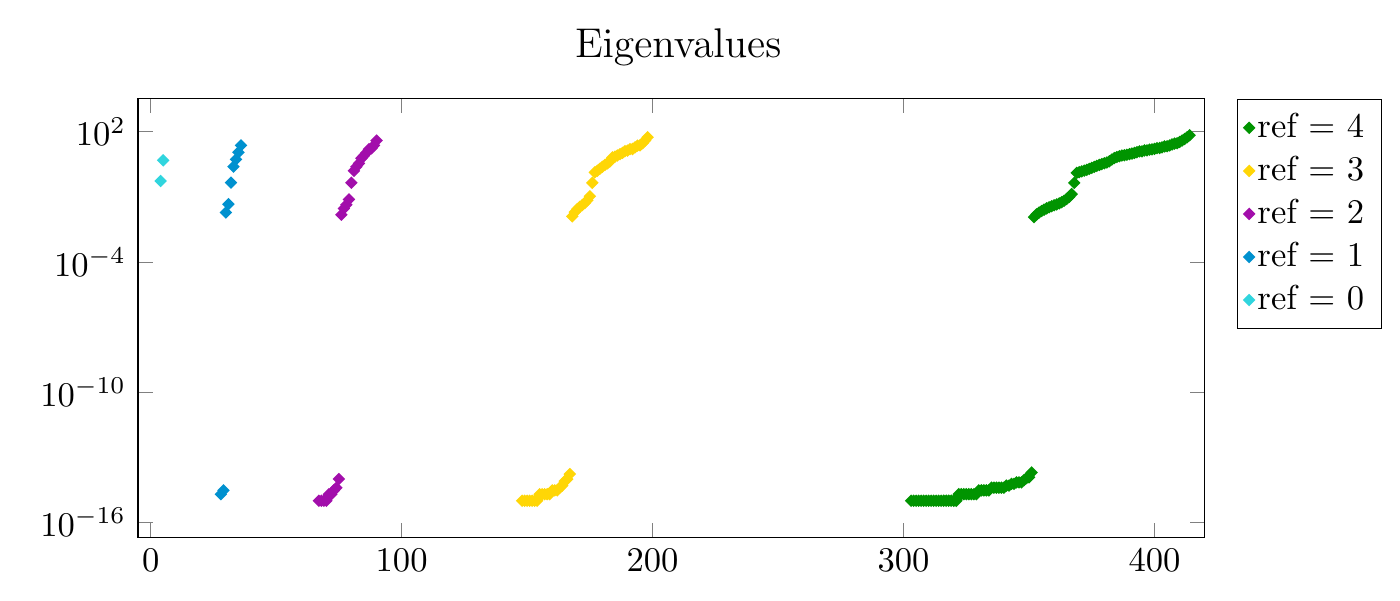}
	\caption{Eigenvalues of the L-shaped domain given in Figure~\ref{fig:LResults} (left) for p = 4 and r = 1,
        and different $h$-refinement levels.}
	\label{fig:EVLshapeCase}
\end{figure}

\section{Conclusions}
\label{conclusions}

In this paper, we have studied $C^1$-smooth isogeometric function spaces over
multi-patch geometries. We have considered geometry
parametrizations composed of multiple patches that meet $C^0$ at patch
interfaces. 
As it is common for isogeometric methods,  the geometry parametrization is
given initially and considered fixed. From that, the geometric continuity conditions are derived. The same conditions need to be satisfied for the graph parametrizations of the isogeometric functions in order to obtain a $C^1$-smooth function space over the given geometry.

We have studied how these conditions affect the traces and the transversal
derivative of the isogeometric function spaces along the patch
interfaces.Indeed, if the trace or transversal derivative at the interface is
over-constrained, the approximation order of the isogeometric function
space is reduced. We identify a class of configurations that allows
for optimal approximation properties of $C^1$ isogeometric spaces, so
called \emph{analysis-suitable $G^1$} (AS $G^1$) geometry
parametrizations. We show  numerically  that AS $G^1$ geometries indeed allow for
optimal approximation.  We have not  developed a complete approximation
theory with classical error estimates, but this  will be  the topic of
a future paper.  Parametrizations that are not AS
$G^1$ cause  suboptimal  order of 
approximation.  In the worst case the convergence 
under $h$-refinement is prohibited, a behaviour that we have named
$C^1$ locking. 

We have addressed mainly the case of  planar B-spline geometries, but 
we have briefly discussed the generalization to  surfaces and NURBS patches. All the
results are supported and confirmed by numerical tests for various
degrees and orders of regularity of the spline space. We have numerically
solved a bilaplacian problem over several AS and
non AS $G^1$  geometries,  by a standard Galerkin
approach. As we pointed out, the numerical implementation of the $C^1$
conditions poses some non-negligible difficulties for complex
configurations of non AS $G^1$ geometries.

An important question that remains to be studied in more detail is the
flexibility of analysis-suitable $G^1$ geometries. As we have shown,
the AS $G^1$ class contains bilinear patches but extends to more
general configurations. We formulate the problem of flexibility in the
following way:  Given a collection of boundary curves, is it possible
to find patches that interpolate the boundary curves and that form an
AS $G^1$ geometry?  For piecewise linear boundary curves, the problem
can be solved by using bilinear patches.  In Figure
\ref{fig:RectangleResults} we have
shown the interesting example of a piecewise biquadratic AS $G^1$ parametrization of
a $C^1$ simply-connected domain. The  extension to arbitrary
degrees and topology  deserves further investigation. 

The construction for
AS $G^1$ surfaces could be more difficult. In the
planar case, the interior parametrization of the patches may be
modified in order to achieve analysis-suitability. This is not
feasible for surfaces, as the surface itself changes if the
parametrization of the interior is changed. AS $G^1$ constructions are
possible if the surface is given as the image of a planar domain. For
the general setting, an explicit construction remains an open problem
to be considered in future research. 

\section*{Acknowledgments}
The authors would like to thank  Pablo Antolin and Massimiliano
Martinelli for fruitful discussions on the topics of the paper, and
the anonymous reviewers for valuable comments and suggestions.
The authors were partially supported by the European Research Council
through the  FP7 ERC Consolidator Grant n.616563 \emph{HIGEOM}, and by the Italian
MIUR through the PRIN  ``Metodologie innovative nella modellistica differenziale numerica''.  This support is gratefully acknowledged.

\section*{Appendix A: AS $G^1$ parametrization of a quarter of circle}

In this appendix we discuss in detail the construction of the geometry
parametrization of the example shown in Figure
\ref{fig:QuartOfCircleResults} in Section
\ref{sec:multi-patch-ASG1}. This is based on the NURBS setting of
Section~\ref{sec:beyond-splines} and on the ideas in
\cite{bercovier2014smooth} as well as \cite[Section
3.4]{kapl-vitrih-juttler-birner-15}.  
Each patch of the three-patch geometry is constructed from a combination of a bilinear mapping 
\begin{equation*}
    	\f B^{(i)}: [0,1]^2 \rightarrow Q^{(i)} \subset \Delta = \{ (u,v) \in [0,1]^2 \, : \, u+v \leq 1\}
\end{equation*}
onto a quadrilateral $Q^{(i)}$ within a reference triangle $\Delta$ and a global mapping 
\begin{equation*}
    	\widetilde{\f G}: \Delta \rightarrow \widetilde\Omega \subset \RR^3
\end{equation*}
from the reference triangle to the surface patch $\widetilde\Omega$ in homogeneous coordinates. These mappings and corresponding domains are shown in Figure~\ref{fig:appendix}. 
\begin{figure}[h]
	\centering
	\includegraphics[width=0.8\textwidth]{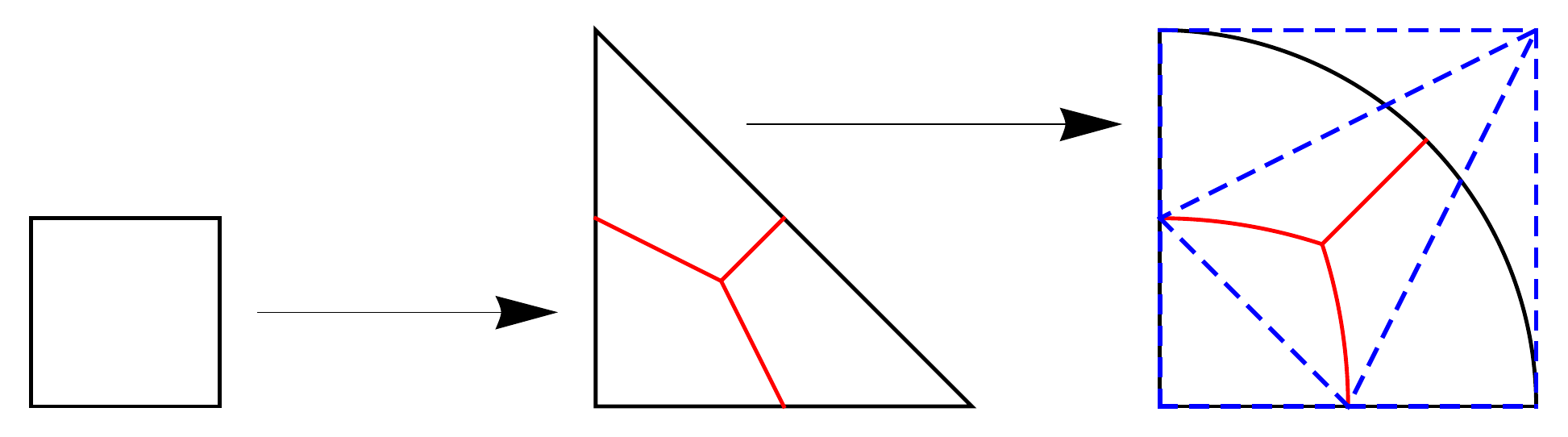}
	\put(-160,60){$\f G$}
	\put(-290,35){$\f B^{(i)}$}
	\put(-360,25){$[0,1]^2$}
	\put(-220,20){$Q^{(i)}$}
	\put(-225,70){$\Delta$}
	\put(-80,70){$\Omega$}
	\caption{Mappings $\f B^{(i)}$ and $\f G$ and corresponding domains.}
	\label{fig:appendix}
\end{figure}
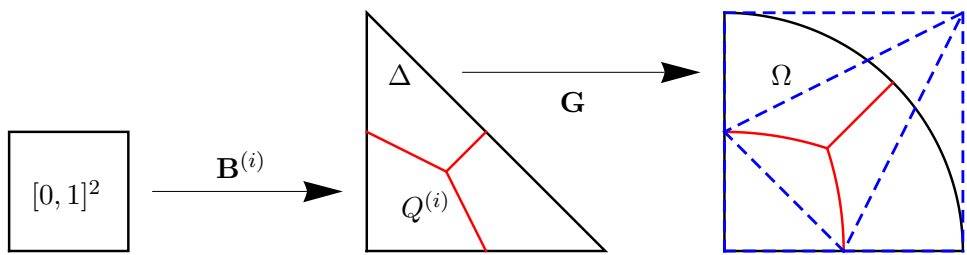
 
Here the mapping $\widetilde{\f G}$ is a triangular B\'ezier patch of total degree $p=2$ with
\begin{equation*}
\widetilde{\f G} (s,t) = \sum_{i+j+k = 2} g_{i,j,k} \; s^i \; t^j \; (1-s-t)^k \; \frac{i ! \, j ! \, k ! }{2},
\end{equation*}
with control points 
\begin{equation*}
\begin{array}{lll}
g_{0,0,2} = (0,0,1)^T & g_{0,1,1} = (0,\sqrt{2},2\sqrt{2})^T & g_{0,2,0} = (0,1,1)^T \\
g_{1,0,1} = (\sqrt{2},0,2\sqrt{2})^T & g_{1,1,0} = (2\sqrt{2},2\sqrt{2},2\sqrt{2})^T &  \\
g_{2,0,0} = (1,0,1)^T & &
\end{array}
\end{equation*}
in homogeneous coordinates. In Figure~\ref{fig:appendix}, the dashed blue lines represent the corresponding control point grids in Cartesian coordinates. Each quadrilateral $Q^{(1)}$, $Q^{(2)}$ and $Q^{(3)}$ is formed by one corner point of the triangle $\Delta$, two adjacent edge midpoints as well as the center of gravity $(1/3,1/3)^T$. The bilinear mapping $\f B^{(i)}$ is (up to rotations of the parameter domain) completely determined by its image $Q^{(i)}$. By construction, the mapping $\f F^{(i)} = \f G \circ \f B^{(i)}$ is a rational bi-quadratic function in each component. 

Using this construction, the theory developed in Sections \ref{surfaces} and \ref{sec:beyond-splines} can be applied to the example in Figure~\ref{fig:QuartOfCircleResults}.

\section*{Appendix B: AS $G^1$ parametrization of a smooth
  simply-connected domain}

In the following we give a description of the domain depicted in Figure \ref{fig:5PatchSmoothResults}. 
The domain is composed of five bi-quadratic patches and the boundary is $C^1$-smooth. The central patch 
is just the shifted unit square $\Omega^{(c)} = [-\frac{1}{2},\frac{1}{2}]^2$ parameterized by 
\begin{equation*}
\f F^{(c)} (u,v) = \left( u-\frac{1}{2}, v-\frac{1}{2} \right)^T.
\end{equation*}
The top patch $\Omega^{(t)}$ is parameterized by 
\begin{equation*}
\f F^{(t)} (u,v) = \left( 
\begin{array}{c}
\left(u-\frac{1}{2}\right)\left( 1 + \frac{\sqrt{17}-3}{2}v - \frac{\sqrt{17}-5}{2}v^2 \right)\\
2v^2\left(u-u^2\right) + \frac{1}{2}\left( 1 + \frac{\sqrt{17}-3}{2}v - \frac{\sqrt{17}-5}{2}v^2 \right)
\end{array}
\right)
\end{equation*}
and the parametrizations $\f F^{(l)}$, $\f F^{(b)}$ and $\f F^{(r)}$ for the left, bottom and right patches are given by rotations of the top patch.


\begin{thebibliography}{10}

\bibitem{bazilevs2006isogeometric}
Y.~Bazilevs, L.~Beirao~da Veiga, J.~A. Cottrell, T.~J.~R. Hughes, and G.~Sangalli.
\newblock Isogeometric analysis: approximation, stability and error estimates
  for h-refined meshes.
\newblock {\em Mathematical Models and Methods in Applied Sciences},
  16(07):1031--1090, 2006.

\bibitem{beeker1986smoothing}
E.~Beeker.
\newblock Smoothing of shapes designed with free-form surfaces.
\newblock {\em Computer-aided design}, 18(4):224--232, 1986.

\bibitem{da2012isogeometric}
L.~Beirao~da Veiga, A.~Buffa, C.~Lovadina, M.~Martinelli, and G.~Sangalli.
\newblock An isogeometric method for the {Reissner--Mindlin} plate bending
  problem.
\newblock {\em Computer Methods in Applied Mechanics and Engineering},
  209:45--53, 2012.

\bibitem{da2011some}
L.~Beirao~da Veiga, A.~Buffa, J.~Rivas, and G.~Sangalli.
\newblock Some estimates for h--p--k-refinement in isogeometric analysis.
\newblock {\em Numerische Mathematik}, 118(2):271--305, 2011.

\bibitem{da2014mathematical}
L.~Beirao~da Veiga, A.~Buffa, G.~Sangalli, and R.~V{\'a}zquez.
\newblock Mathematical analysis of variational isogeometric methods.
\newblock {\em Acta Numerica}, 23:157--287, 2014.

\bibitem{da2012anisotropic}
L.~Beirao~da Veiga, D.~Cho, and G.~Sangalli.
\newblock Anisotropic {NURBS} approximation in isogeometric analysis.
\newblock {\em Computer Methods in Applied Mechanics and Engineering},
  209:1--11, 2012.

\bibitem{benson2011large}
D.~J. Benson, Y.~Bazilevs, M.-C. Hsu, and T.~J.~R. Hughes.
\newblock A large deformation, rotation-free, isogeometric shell.
\newblock {\em Computer Methods in Applied Mechanics and Engineering},
  200(13):1367--1378, 2011.

\bibitem{bercovier2014smooth}
M.~Bercovier and T.~Matskewich.
\newblock Smooth {Bezier} surfaces over arbitrary quadrilateral meshes.
\newblock {\em arXiv preprint arXiv:1412.1125}, 2014.

\bibitem{brivadis2015isogeometric}
E.~Brivadis, A.~Buffa, B.~Wohlmuth, and L.~Wunderlich.
\newblock Isogeometric mortar methods.
\newblock {\em Computer Methods in Applied Mechanics and Engineering},
  284:292--319, 2015.

\bibitem{Buchegger2015}
F.~Buchegger, B.~J{\"u}ttler, and A.~Mantzaflaris.
\newblock Adaptively refined multi-patch {B}-splines with enhanced smoothness.
\newblock {\em Applied Mathematics and Computation}, page in press, 2015.

\bibitem{cirak2002integrated}
F.~Cirak, M.~J. Scott, E.~K. Antonsson, M.~Ortiz, and P.~Schr{\"o}der.
\newblock Integrated modeling, finite-element analysis, and engineering design
  for thin-shell structures using subdivision.
\newblock {\em Computer-Aided Design}, 34(2):137--148, 2002.

\bibitem{cottrell2009isogeometric}
J.~A. Cottrell, T.~J.~R. Hughes, and Y.~Bazilevs.
\newblock {\em Isogeometric analysis: toward integration of {CAD} and {FEA}}.
\newblock Wiley, 2009.

\bibitem{evans2009n}
J.~A. Evans, Y.~Bazilevs, I.~Babu{\v{s}}ka, and T.~J.~R. Hughes.
\newblock n-{W}idths, sup--infs, and optimality ratios for the k-version of the
  isogeometric finite element method.
\newblock {\em Computer Methods in Applied Mechanics and Engineering},
  198(21):1726--1741, 2009.

\bibitem{gomez2008isogeometric}
H.~G{\'o}mez, V.~M. Calo, Y.~Bazilevs, and T.~J.~R. Hughes.
\newblock Isogeometric analysis of the {Cahn--Hilliard} phase-field model.
\newblock {\em Computer Methods in Applied Mechanics and Engineering},
  197(49):4333--4352, 2008.

\bibitem{gomez2010isogeometric}
H.~G{\'o}mez, T.~J.~R. Hughes, X.~Nogueira, and V.~M. Calo.
\newblock Isogeometric analysis of the isothermal {Navier--Stokes--Korteweg}
  equations.
\newblock {\em Computer Methods in Applied Mechanics and Engineering},
  199(25):1828--1840, 2010.

\bibitem{groisser2015matched}
D.~Groisser and J.~Peters.
\newblock Matched ${G}^k$-constructions always yield ${C}^k$-continuous
  isogeometric elements.
\newblock {\em Computer aided geometric design}, 34:67--72, 2015.

\bibitem{hughes2005isogeometric}
T.~J.~R. Hughes, J.~A. Cottrell, and Y.~Bazilevs.
\newblock Isogeometric analysis: {CAD}, finite elements, {NURBS}, exact
  geometry and mesh refinement.
\newblock {\em Computer methods in applied mechanics and engineering},
  194(39):4135--4195, 2005.

\bibitem{hughes2008duality}
T.~J.~R. Hughes, A.~Reali, and G.~Sangalli.
\newblock Duality and unified analysis of discrete approximations in structural
  dynamics and wave propagation: comparison of p-method finite elements with
  k-method {NURBS}.
\newblock {\em Computer methods in applied mechanics and engineering},
  197(49):4104--4124, 2008.

\bibitem{juttler2015isogeometric}
B.~J{\"u}ttler, A.~Mantzaflaris, R.~Perl, and M.~Rumpf.
\newblock On isogeometric subdivision methods for {PDE}s on surfaces.
\newblock {\em NFN Technical Report No. 27}, 2015.

\bibitem{kapl-buchegger-bercovier-juttler-16}
M.~Kapl, F.~Buchegger, M.~Bercovier, and B.~J{\"u}ttler.
\newblock Isogeometric analysis with geometrically continuous functions on
  multi-patch geometries.
\newblock {\em NFN Technical Report No. 35}, 2015.

\bibitem{kapl-vitrih-juttler-birner-15}
M.~Kapl, V.~Vitrih, B.~J{\"u}ttler, and K.~Birner.
\newblock Isogeometric analysis with geometrically continuous functions on
  two-patch geometries.
\newblock {\em Computers \& Mathematics with Applications}, 70(7):1518--1538,
  2015.

\bibitem{kiendl-bazilevs-hsu-wuechner-bletzinger-10}
J.~Kiendl, Y.~Bazilevs, M.-C. Hsu, R.~W{\"u}chner, and K.-U. Bletzinger.
\newblock The bending strip method for isogeometric analysis of
  {K}irchhoff-{L}ove shell structures comprised of multiple patches.
\newblock {\em Computer Methods in Applied Mechanics and Engineering},
  199(35):2403--2416, 2010.

\bibitem{kiendl-bletzinger-linhard-09}
J.~Kiendl, K.-U. Bletzinger, J.~Linhard, and R.~W{\"u}chner.
\newblock Isogeometric shell analysis with {K}irchhoff-{L}ove elements.
\newblock {\em Computer Methods in Applied Mechanics and Engineering},
  198(49):3902--3914, 2009.

\bibitem{Kleiss2012}
S.~K. Kleiss, C.~Pechstein, B.~J{\"u}ttler, and S.~Tomar.
\newblock {IETI} - isogeometric tearing and interconnecting.
\newblock {\em Computer Methods in Applied Mechanics and Engineering},
  247-248(0):201 -- 215, 2012.

\bibitem{liu1989gc}
D.~Liu and J.~Hoschek.
\newblock {$GC 1$} continuity conditions between adjacent rectangular and
  triangular {B{\'e}zier} surface patches.
\newblock {\em Computer-Aided Design}, 21(4):194--200, 1989.

\bibitem{lyche2014hermite}
T.~Lyche and G.~Muntingh.
\newblock A {H}ermite interpolatory subdivision scheme for {$C^2$}-quintics on
  the {P}owell--{S}abin 12-split.
\newblock {\em Computer Aided Geometric Design}, 31(7):464--474, 2014.

\bibitem{Matskewich-PhD}
T.~Matskewich.
\newblock {\em Construction of C1 surfaces by assembly of quadrilateral patches
  under arbitrary mesh topology}.
\newblock PhD thesis, Hebrew University of Jerusalem, 2001.

\bibitem{mourrain2015geometrically}
B.~Mourrain, R.~Vidunas, and N.~Villamizar.
\newblock Geometrically continuous splines for surfaces of arbitrary topology.
\newblock {\em arXiv preprint arXiv:1509.03274}, 2015.

\bibitem{nguyen2014comparative}
T.~Nguyen, K.~Kar{\v{c}}iauskas, and J.~Peters.
\newblock A comparative study of several classical, discrete differential and
  isogeometric methods for solving {Poisson}'s equation on the disk.
\newblock {\em Axioms}, 3(2):280--299, 2014.

\bibitem{nguyen2016}
T.~Nguyen, K.~Kar{\v{c}}iauskas, and J.~Peters.
\newblock C1 finite elements on non-tensor-product 2d and 3d manifolds.
\newblock {\em Applied mathematics and computation}, 272:148--158, 2016.

\bibitem{pauletti-martinelli-cavallini-14}
M.~S. Pauletti, M.~Martinelli, N.~Cavallini, and P.~Antol\`in.
\newblock Igatools: {An Isogeometric Analysis Library}.
\newblock {\em SIAM Journal on Scientific Computing}, 2015.

\bibitem{peters-PhD}
J.~Peters.
\newblock {\em Fitting smooth parametric surfaces to {3D} data}.
\newblock PhD thesis, University of Wisconsin, 1990.

\bibitem{peters1990}
J.~Peters.
\newblock Smooth mesh interpolation with cubic patches.
\newblock {\em Computer-Aided Design}, 22(2):109--120, 1990.

\bibitem{peters1991smooth}
J.~Peters.
\newblock Smooth interpolation of a mesh of curves.
\newblock {\em Constructive Approximation}, 7(1):221--246, 1991.

\bibitem{peters-handbook}
J.~Peters.
\newblock Geometric continuity.
\newblock In {\em Handbook of Computer Aided Geometric Design}, pages 193--229.
  Elsevier, 2002.

\bibitem{sangalli2015unstructured}
G.~Sangalli, T.~Takacs, and R.~V{\'a}zquez.
\newblock Unstructured spline spaces for isogeometric analysis based on spline
  manifolds.
\newblock {\em arXiv preprint arXiv:1507.08477}, 2015.

\bibitem{scottPhD}
M.~A. Scott.
\newblock {\em T-splines as a Design-Through-Analysis Technology}.
\newblock PhD thesis, The University of Texas at Austin, 2011.

\bibitem{scott2013isogeometric}
M.~A. Scott, R.~N. Simpson, J.~A. Evans, S.~Lipton, S.~P.~A. Bordas, T.~J.~R.
  Hughes, and T.~W. Sederberg.
\newblock Isogeometric boundary element analysis using unstructured
  {T-splines}.
\newblock {\em Computer Methods in Applied Mechanics and Engineering},
  254:197--221, 2013.

\bibitem{scott2014isogeometric}
M.~A. Scott, D.~C. Thomas, and E.~J. Evans.
\newblock Isogeometric spline forests.
\newblock {\em Computer Methods in Applied Mechanics and Engineering},
  269:222--264, 2014.

\bibitem{speleers2012isogeometric}
H.~Speleers, C.~Manni, F.~Pelosi, and M.~L. Sampoli.
\newblock Isogeometric analysis with {P}owell--{S}abin splines for
  advection--diffusion--reaction problems.
\newblock {\em Computer Methods in Applied Mechanics and Engineering},
  221:132--148, 2012.

\bibitem{Takacs:2015aa}
S.~Takacs and T.~Takacs.
\newblock Approximation error estimates and inverse inequalities for
  {B-splines} of maximum smoothness.
\newblock {\em Mathematical Models and Methods in Applied Sciences},
  26(7):1411--1445, 2016.

\bibitem{Takacs2014}
T.~Takacs, B.~J{\"u}ttler, and O.~Scherzer.
\newblock Derivatives of isogeometric functions on n-dimensional rational
  patches in ${R}^d$.
\newblock {\em Computer Aided Geometric Design}, 31(7-8):567 -- 581, 2014.
\newblock Recent Trends in Theoretical and Applied Geometry.

\bibitem{wu2015bicubic}
M.~Wu, B.~Mourrain, A.~Galligo, and B.~Nkonga.
\newblock Bicubic spline spaces over rectangular meshes with arbitrary
  topologies.
\newblock {\em hal-unice.archives-ouvertes.fr}, 2015.

\end{thebibliography}
\end{document}